\documentclass[10pt]{amsart}
\usepackage{graphicx}
\usepackage{amscd}
\usepackage{amsmath}
\usepackage{amsthm}
\usepackage{amsfonts}
\usepackage{amssymb}
\usepackage{mathrsfs}
\usepackage{bm}
\usepackage{enumerate}
\usepackage{amsrefs}
\usepackage{xcolor}
\usepackage[colorlinks, citecolor=blue, linkcolor=red, pdfstartview=FitB]{hyperref}

\usepackage{marginnote}	
\usepackage{soul} 			

\newcommand{\bB}{{\mathbb{B}}}
\newcommand{\bC}{{\mathbb{C}}}
\newcommand{\bD}{{\mathbb{D}}}

\newcommand{\bM}{{\mathbb{M}}}

\newcommand{\bR}{{\mathbb{R}}}
\newcommand{\bS}{{\mathbb{S}}}
\newcommand{\bT}{{\mathbb{T}}}

  \newcommand{\A}{{\mathcal{A}}}

  \newcommand{\F}{{\mathcal{F}}}
  
\renewcommand{\H}{{\mathcal{H}}}

  \newcommand{\M}{{\mathcal{M}}}

  \newcommand{\T}{{\mathcal{T}}}

\newcommand{\rC}{\mathrm{C}}

\newcommand{\eps}{\varepsilon}
\renewcommand{\phi}{\varphi}

\newcommand{\upchi}{{\raise.35ex\hbox{$\chi$}}}

\newcommand{\ol}{\overline}

\newcommand{\qand}{\quad\text{and}\quad}

\newcommand{\at}{\operatorname{at}}

\newcommand{\AC}{\operatorname{AC}}
\newcommand{\SG}{\operatorname{SG}}

\newcommand{\diag}{\operatorname{diag}}

\newtheorem{lemma}{Lemma}[section]
\newtheorem{theorem}[lemma]{Theorem}
\newtheorem{proposition}[lemma]{Proposition}
\newtheorem{corollary}[lemma]{Corollary}
\newtheorem{theoremx}{Theorem}

\theoremstyle{definition}

\newtheorem{example}{Example}

\date{\today}
\author{David P. Blecher}

\address{Department of Mathematics, University of Houston}

\email{dpbleche@central.uh.edu\vspace{-2ex}}

\author{Rapha\"el Clou\^atre}

\address{Department of Mathematics, University of Manitoba, Winnipeg, Manitoba, Canada R3T 2N2}

\email{raphael.clouatre@umanitoba.ca\vspace{-2ex}}

\thanks{R.C. was partially supported by an NSERC Discovery Grant. D.B. was partially supported by  a Simons Foundation Travel/Collaboration grant
and NSF Grant DMS-2154903.}


\title[Null projections and noncommutative function theory]{Null projections and noncommutative function theory in operator algebras}
\begin{document}
\begin{abstract}
We study projections in the bidual of a $\rC^*$-algebra $B$ that are null with respect to a subalgebra $A$, that is projections $p\in B^{**}$ satisfying $|\phi|(p)=0$ for every $\phi\in B^*$ annihilating $A$. In the separable case, $A$-null  projections are precisely the 
peak projections in the bidual of $A$ at which the subalgebra $A$ interpolates the entire $\rC^*$-algebra $B$.  These are analogues of null sets in classical function theory, on which several profound results rely. This motivates the development of a noncommutative variant, which we use to find appropriate `quantized' versions of some of these classical facts. 
Through a delicate generalization of a theorem of Varopoulos, we show that, roughly speaking, sufficiently regular interpolation projections are null precisely when their atomic parts are. As an application, we give alternative proofs and sharpenings of some recent peak
interpolation results of Davidson and Hartz for algebras on Hilbert function spaces, also illuminating thereby how earlier noncommutative peak-interpolation theory may be applied.
 In another direction, given a convex subset of the state space of $B$, we characterize when the associated Riesz projection is null. This is then applied to various important topics in noncommutative function theory, such as the F.\& M. Riesz property, the existence of Lebesgue decompositions, the description of Henkin functionals, and Arveson's noncommutative 
 Hardy spaces (maximal subdiagonal algebras).  
\end{abstract} 
\maketitle

\section{Introduction}

Let $B$ be a unital $\rC^*$-algebra. A projection $u$ in the von Neumann algebra $B^{**}$ is \emph{open} if it is the increasing weak-$*$ limit of a net of positive contractions in $B$. Likewise, a projection $q\in B^{**}$ is closed if $1-q$ is open. These notions of non-commutative topology were introduced and developed by Akemann  \cite{akemann1969, akemann1970left}. 

Next, let $A\subset B$ be a unital norm-closed subalgebra. A closed projection $q\in B^{**}$ is said to be an \emph{$A$-peak} projection if there is a contraction $a\in A$ such that $aq=q$ and $\|ap\|<1$ for every closed projection $p\in B^{**}$ orthogonal to $q$. This is readily seen to be a generalization of the classical idea of a peak set for a function algebra \cite{gamelin1969}. The foundations of the theory of non-commutative peaking in this sense have been laid out in a series of papers by the first author and collaborators (see e.g.\ \cite{hay2007,BHN2008,BR2011,BN2012,BR2013,blecher2013}, the last of these being in part a survey).  We shall refer to this theory 
below as {\em noncommutative  peak-interpolation}. 
An important result proved in these papers is that, when $B$ is separable, a closed projection $q$ is $A$-peak precisely when $q$ lies in $A^{\perp\perp}$. Outside of the separable setting,   
closed projections in $A^{\perp\perp}$ are the infima of  $A$-peak projections.  These are called the {\em generalized $A$-peak projections} (or $p$-{\em projections}) and they have properties similar to the peak projections. 

Peak sets are a key tool in classical function theory.
  For example, when $B=\rC(X)$ is separable, it is known that the collection of $A$-peak points (i.e.\  the minimal $A$-peak projections in $B^{**}$) coincides with the minimal boundary of $A$ \cite{bishop1959, bishop1959bdry}. 
By analogy, in the general setting one may still hope to extract structural information about  $A$ by studying the closed projections in $A^{\perp\perp}$, and this was 
initiated in the work referred to above. In more recent years, this paradigm has been successfully exploited in the context of several foundational purely operator algebraic problems \cite{CTh2022,CTh2023,clouatre2023pure}.

The current paper is a contribution to this general program. Herein we focus on a subclass of the projections in $A^{\perp\perp}$ that assume the role that measurable sets satisfying $|\mu|(E)=0$ (for some natural set of measures $\mu$) play in classical function theory.
Several profound results  and key techniques  in classical function theory hinge on properties of such sets, and their relation to peak sets;   see for example  Chapters 9 and 10 of \cite{rudin2008}.   Our 
goal is to develop the noncommutative variant in order to obtain `quantized' analogues of these 
deep facts for use in more general algebras.

We say that a projection $q\in B^{**}$ is \emph{$A$-null} 
if $|\phi|(q)=0$ for every $\phi\in A^{\perp}$. Here and throughout, given a continuous linear function $\phi\in B^*$, we denote by $|\phi|\in B^*$ its modulus.
 It follows that $A$-null projections automatically lie in $A^{\perp\perp}$. Several characterizations of null projections are proved in Section \ref{S:null}.

When $q$ satisfies that 
$qA=qB$ as subspaces in $B^{**}$, we say that $q$ is an \emph{$A$-interpolation} projection. Note that this is a one-sided concept, in that it is not equivalent to the equality 
$Aq = Bq$ in general (as one may see even in the upper triangular $2 \times 2$ matrices). 
When $B$ is commutative, closely related notions of interpolation have long been of interest to function theorists, even for countable sets (see for instance \cite{AHMR2019interp} and the references therein). 

It will be helpful to highlight some of the relationship between $q$ being an $A$-null projection, and $q$ being a (generalized) peak projection as well as an $A$-interpolation projection.  We emphasize the following  result, essentially  proved in \cite{CT2023Henkin}. 

\bigskip 

\noindent {\bf Basic principle:} Let $B$ be a $\rC^*$-algebra 
 and let $A \subset B$ be a norm-closed subalgebra. If $q$ is a closed projection in $B^{**}$, 
 then the following statements are equivalent.
\begin{enumerate}[{\rm (i)}] 
\item $q$ is $A$-null.
\item $q \in A^{\perp\perp}$ and 
$qA = qB$.
\end{enumerate}

\begin{proof} 
(ii) $\Rightarrow$(i): This is proved in  \cite[Theorem 6.2]{CT2023Henkin} when $q$ is closed and $B$ unital, but the same proof works in general  (it is not necessary that $q$ be closed).

(i) $\Rightarrow$(ii): (This does not require that $A$ be a subalgebra.) That $q \in A^{\perp\perp}$ follows from the inequality 
 $|\phi(q)|^2 \leq |\phi|(q) \, |\phi|(1) $ if $\phi \in A^{\perp}$ (see Lemma \ref{L:stateineq} 
below). 
That 
$qA = qB$ is proved in the unital case in the first lines of the proof of \cite[Proposition 3.4]{hay2007}.  The nonunital case follows by the same proof 
 (which proof is given with more detail in  the proof that (ii) implies (iii) in \cite[Theorem 6.2]{CT2023Henkin}) 
once we know that 
$qA$ is closed.   

To see this, we simply need to adapt the proof  of \cite[Proposition 3.1]{hay2007} to the case where $X=A$ and $B$ is nonunital. Consider $B^1=B+\bC 1\subset B^{**}$. Following the argument in \cite{hay2007}, we find a decreasing net $(e_t)$ in $B^1$ converging to $q$ in the weak-$*$ topology of $B^{**}$. While $e_t$ may not belong to $B$, we nevertheless have 
$Be_t  \subset B$ for each index $t$, which is all that is needed in Hay's argument to conclude that $qA$ is isometric to the quotient of $A$ by some norm-closed subspace, and hence is closed.
\end{proof}

In fact, much more is true:  for an $A$-null closed projection $q$,  the peaking and interpolation phenomena can be arranged to occur simultaneously in a very strong sense, with 
many extra features -- this is called \emph{peak-interpolation}
(see e.g.\ Bishop \cite{bishop1962}).   We refer to 
e.g.\ \cite{blecher2013, blecher2022} for background on noncommutative peaking and interpolation, although the latter paper focuses on a stronger two-sided version of
 peak-interpolation.

In light of the rigidity that the peak-interpolation phenomenon imposes on $q$, one may guess that actually verifying that $q$ be $A$-null can be difficult in practice. One of our main results addresses this issue, drawing inspiration 
from  a classical theorem of Varopoulos, which we now recall. Let $X$ be a compact Hausdorff space and let $A\subset \rC(X)$ be a closed unital subalgebra. In \cite{varopoulos1971}, it is proved that in order for a closed subset $K\subset X$ to be $A$-null, it suffices that $K$ be an $A$-interpolation set consisting entirely of $A$-peak points. 

Generalizing Varopoulos' theorem for closed $A$-interpolation projections when $B$ is an arbitrary $\rC^*$-algebra will occupy us for the subsequent two sections of the paper. Section \ref{S:indist} deals with non-commutative topological matters that are required to properly frame our result. We seek to understand to which extent a projection in $B^{**}$ is determined by the universally measurable projections dominating it. This prompts us to introduce the notion of \emph{measurably indistinguishable} projections, which is crucial for our generalization. The key observation we make is Proposition \ref{P:induniv}, showing that universally measurable projections are measurably indistinguishable from their so-called \emph{atomic parts}.

Section \ref{S:ncVar} is concerned with our noncommutative version of Varopoulos' theorem. For this purpose, we need the following ``complete" refinement of the notion of interpolation projection.  
For $\kappa>0$, we say that a projection $q$ is a (two-sided) \emph{complete $(A,\kappa)$-interpolation projection} if, given any $n\geq 1$ and any matrix $[b_{ij}]\in \bM_n(B)$, there are matrices $[a_{ij}],[a'_{ij}]\in \bM_n(A)$ such that 
\[
 a_{ij}q=b_{ij}q ,\quad  q a'_{ij}=qb_{ij} \quad \text{ for every } 1\leq i,j\leq n
\]
and satisfying
\[
 \|[a_{ij}]\|\leq \kappa \|[b_{ij}]\|,\quad \|[a'_{ij}]\|\leq \kappa \|[b_{ij}]\|.
\]
When $B$ is a nuclear $\rC^*$-algebra, which is the case for the main applications we have in mind, we show in Proposition \ref{P:interpnuclear} that any closed central $A$-interpolation projection must in fact be a complete $(A,\kappa)$-interpolation projection for some $\kappa>0$.
Before we can state 
the next theorem we recall that a projection $p\in B^{**}$ is said to be \emph{$A$-invariant} if $ap=pap$ for each $a\in A$. 
One of the main 
technical results of this paper is Theorem \ref{T:ncVar}, which goes as follows.

\begin{theoremx}\label{T:A}
 Let $B$ be a unital $\rC^*$-algebra and let $A\subset B$ be a unital norm-closed subalgebra.  Let $q\in B^{**}$ be a projection with the following properties:
 \begin{enumerate}[{\rm(i)}]
 \item $q$ is closed, 
 \item there is $\kappa>0$ for which $q$ is a complete $(A,\kappa)$-interpolation projection,
 \item $q$ commutes with $A$, and
 \item the atomic part of $q$ lies in $A^{\perp\perp}$.
\end{enumerate}
Then, for every $\phi\in A^\perp$  and every open $A$-invariant projection $u\in B^{**}$ dominating $q$, we have
\[
 |\phi|(q)\leq  \kappa^3 \sqrt{|\phi|(1)                                                                                                                                                                                                                                                                                                                                             \, | \phi|(u-q)}.
\]
\end{theoremx}

We interpret Theorem \ref{T:A} as saying that for sufficiently 
`topologically regular' interpolation projections
(i.e.\ where such $u$ may approach $q$)
commuting with $A$, the property of being $A$-null is entirely captured by the atomic part.  In Proposition \ref{P:CXD}, we exhibit a non-trivial class 
of algebras where the atomic part of every projection lies in $A^{\perp\perp}$. In Corollary \ref{C:ncVarcomm}, we identify some natural sufficient conditions under which the conclusion of Theorem \ref{T:A} can be strengthened to $q$ being null. In particular, we show in Corollary \ref{C:Varopoulos} how to recover Varopoulos' original theorem using our results.

In recent years it has been of interest to produce peak-interpolation theorems in the context of  multiplier algebras of Hilbert function spaces; see for instance 
\cite{CD2016duality} and \cite{DH2023}. Despite there being  some obvious conceptual overlaps between these research directions and the 
earlier general noncommutative peak-interpolation theory, no explicit connection was hitherto drawn in the literature between these two topics. In particular, the 
extensive  earlier 
noncommutative peak-interpolation machinery was not 
  brought to bear in proving results in \cite{CD2016duality} and \cite{DH2023}.  
  In Section \ref{S:A(H)}, we illustrate how this can be done. The key observation making this possible is found in Theorem \ref{T:nullAH}, where we establish a
   bijection between null projections and certain small sets on the unit sphere.  As an application of Theorem \ref{T:A} and of the noncommutative peak-interpolation theory, we then give an alternative proof for several of the results found in \cite{DH2023}. In fact, we achieve a sharper conclusion, since we show that peak-interpolation can be implemented  while maintaining domination by a control function, in line with Bishop's original theorem \cite{bishop1962}. We call this \emph{Bishop peak-interpolation}.  
  
  We omit the precise definitions for now, and refer the reader to the relevant section for more detail. Let $\H$ be a Hilbert space of analytic functions on the open unit ball $\bB_d\subset \bC^d$, admitting  a reproducing kernel that is unitarily invariant and regular. There is a central projection $z\in \T(\H)^{**}$ and a natural $*$-isomorphism $\Theta: \T(\H)^{**}z\to  \rC(\bS_d)^{**}$.

\begin{theoremx} \label{T:B} 
For a closed projection $q \in \T(\H)^{**}$ the following are equivalent.
\begin{enumerate}[{\rm (i)}]
\item $q$ is an $\A(\H)$-peak projection.
\item $q$ is $\A(\H)$-null.
\item $q$ is dominated by $z$ and there is a closed $\M(\H)$-totally null subset $E\subset \bS_d$ such that $\Theta(q)=\chi_E$.
\item $q$ is dominated by $z$ and there is a closed Bishop peak-interpolation subset $E\subset \bS_d$ such that $\Theta(q)=\chi_E$.
\end{enumerate} 
If there exist non-empty $\M(\H)$-totally null sets, then these statements are further equivalent to the following.
\begin{enumerate}
\item[{\rm (v)}]  $q$ is dominated by $z$ and there is a closed $\A(\H)$-interpolation set $E\subset \bS_d$ such that $\Theta(q)=\chi_E$.
\end{enumerate}
 Moreover $\Theta$ restricts to an order preserving  bijection from the 
 projections $q$ as above 
  to the  $\M(\H)$-totally null  sets $E$ (equivalently, Bishop peak-interpolation sets or  $\A(\H)$-peak sets).
  \end{theoremx}

Beyond this result,  our approach can be interpreted as 
a general method that we hope will be applicable in other contexts to 
 prove strong peak-interpolation results. 
 
In Section \ref{S:riesznull}, we turn to a different problem regarding $A$-null projections.
Let $\Delta$ be a norm-closed convex set of states on $B$.
We let $\AC(\Delta)\subset B^*$ denote the set of those functionals that are absolutely continuous 
(see Section \ref{S:riesznull} for the definition of this) with respect to some element of $\Delta$.  
It is known \cite[Theorem 3.5]{CT2023Henkin} that  there is a unique projection $r_\Delta\in B^{**}$ with the property that $\AC(\Delta)^\perp=(1-r_\Delta)B^{**}$. We call this projection the \emph{Riesz projection} corresponding to $\Delta$. We aim to characterize when this Riesz projection satisfies that $1-r_\Delta$ is $A$-null. 

More precisely, our main result relates this condition to other natural properties inspired by classical results in function theory. We say that $A$ has the \emph{F.\&M. Riesz property} in $B$ with respect to $\Delta$ if $\phi( r_\Delta\cdot )\in A^\perp$ for every $\phi\in A^\perp$. This property has been a recurring theme (under various guises) in the operator algebras literature; see for instance \cite{BL2007,CH2025} and the references therein.  

In addition, we say that $A$ has the \emph{Forelli property} in $B$ with respect to $\Delta$ if for every projection $q\in \AC(\Delta)^\perp$  that is the supremum of a collection of closed projections,  there is a contractive net $(a_i)$ in $A$ such that $(q a_i )$ converges to $0$ in the weak-$*$ topology of $B^{**}$, and $(\phi(1-a_i))$ converges to $0$ for every $\phi\in \AC(\Delta)$.  This technical approximation property is directly inspired by \cite[Lemma 9.5.5]{rudin2008}. Our 
next main result can be streamlined as follows.

\begin{theoremx}\label{T:C}
Assume that $\Delta$ is closed in the weak-$*$ topology of $B^{**}$. Consider the following statements.
\begin{enumerate}[{\rm (i)}]
\item The projection $1-r_\Delta$ is $A$-null.
\item All closed projections in $\AC(\Delta)^\perp$ are dominated by a projection in $\AC(\Delta)^\perp\cap A^{\perp\perp}$.
\item $A$ has the Forelli property in $B$ with respect to $\Delta$.
\item $A$ has the F.\&M. Riesz property in $B$ with respect to $\Delta$.
\end{enumerate}
Then,
\[
{\rm (i)}\Rightarrow {\rm (ii)}\Leftrightarrow {\rm (iii)}\Leftrightarrow {\rm (iv)}.
\]
\end{theoremx}
 
The condition that $\Delta$ be weak-$*$ closed is used in proving Lemma \ref{L:Rainwater}, a non-commutative analogue of a classical lemma due to Rainwater \cite{rainwater1969}. Without it, the topological nature of the Riesz projection is not as clear. Nevertheless, many of the implications above remain valid if $\Delta$ is merely assumed to be norm-closed (see Theorem \ref{T:ForelliFM}). We apply these findings in Section \ref{S:Leb} to the existence of Lebesgue decompositions for operator algebras, which is a main theme in the recent paper 
of the second author and Hartz \cite{CH2025}. 
In Corollary \ref{C:Lebdecomp}, we give a new sufficient condition for such decompositions to exist.

In the last two sections of the paper, we aim to understand when all the statements in Theorem 
\ref{T:C} are in fact equivalent. Our first approach, used in Section \ref{S:FM}, consists of comparing the F.\& M. Riesz property defined above with a stronger variant of it. The key result in this direction is Corollary \ref{C:FMs}, showing that the so-called \emph{strong} F.\& M. Riesz property can be reformulated in terms of a band property. 
This is in fact a consequence of a certain rigidity phenomenon (Theorem \ref{T:colerange}) resulting from the F.\& M. Riesz property, 
and which is related to the classical Cole-Range theorem \cite[Theorem 9.6.1]{rudin2008}.
This in turn  is  shown in Examples \ref{E:Henkin} and \ref{E:AB} to unify some previously known results on the structure of Henkin functionals \cite{CD2016duality,CH2025,CT2023Henkin}.
The property of the set of Henkin functionals forming a band plays an important role in these sections.  For example our analysis shows that much of 
the function theory developed  in Chapter 9 of \cite{rudin2008}, and in particular the Cole-Range theorem, rests on two main  pillars: that 
the Henkin functionals form a
band, and the F.\& M. Riesz property.  Also,  
our work  on the (strong) F.\& M.  Riesz property may be viewed as generalizations of Bishop's insight \cite{bishop1962} that a (strong) F.\& M.  Riesz theorem implies a generalized Rudin-Carleson theorem (the latter corresponds for us  to the appropriate projection(s) being $A$-null).  

Finally, in Section \ref{S:AL}, we tackle the equivalence question for Theorem 
\ref{T:C} using a different route, centered around the following notion. We say that $A$ has the \emph{Amar--Lederer property in $B$ with respect to 
$\Delta$} if all closed projections in $\AC(\Delta)^\perp$ are dominated by a closed projection in $\AC(\Delta)^\perp\cap A^{\perp\perp}$.  A very  important 
 classical result of Amar--Lederer \cite{AL1971} shows that $H^\infty(\bD)$ satisfies this property inside $L^\infty(\bT)$ with respect to the state of integration against arclength measure: speaking slightly loosely they show that any  Lebesgue null Borel set is contained in a (closed) Lebesgue null peak set. 
On the other hand, the statements in Theorem 
\ref{T:C}
 fail to be equivalent in this case (Example \ref{E:AL}). The rest of the section is devoted to proving Theorem \ref{T:ALstate}, where we contribute to the theory of maximal subdiagonal algebras. Indeed, we refine known results from \cite{ueda2009,BL2018} using some new noncommutative measure theory, 
 and show that such algebras do satisfy the full Amar and Lederer property for all null closed (indeed `regular') projections, and hence they also have the Forelli property.   
 We say `full' because of a severe restriction imposed  in the latter papers on the null projections used.

We close this introduction by mentioning some notation and background facts.  An operator algebra is a norm closed algebra of operators on a Hilbert space, 
or equivalently a closed subalgebra of a $\rC^*$-algebra. We refer to 
 \cite{BLM2004} for their theory.  We will use silently the fact from basic functional analysis that
$X^{\perp \perp}$ is the weak* closure in $Y^{**}$ of
a subspace  $X \subset
Y$, and is isometric to $X^{**}$.   For us a {\em projection}
is always an orthogonal projection.   The bidual of an operator algebra $A$ is an operator algebra, indeed is a von Neumann algebra
if $A$ is selfadjoint. 
We write $\chi_E$ for
the characteristic function of a set $E$.   In the 
case that $B = \rC(K)$, and $E$ is an open or  closed set
 in $K$, the projection $q = \chi_E$  may be viewed as an element of $\rC(K)^{**}$ in a natural way since 
$\rC(K)^*$ is a certain space of measures on $K$.     Thus if $B = \rC(K)$
the open or closed projections are precisely the   characteristic functions of  open or  closed sets.
Open projections arise naturally in functional analysis. For example, they come naturally out of the ‘spectral theorem/functional calculus': the spectral projections in
$B^{**}$ of a self-adjoint operator $a \in B$ corresponding to open sets in the spectrum of $T$ are open projections. 
 The nullspace of a normal state on a von Neumann algebra is closed under `joins' and `meets' of projections.

\section{Null projections}\label{S:null}

Let $B$ be a $\rC^*$-algebra and let $M\subset B$ be a subspace. Recall that a projection $p\in B^{**}$ is \emph{$M$-null} if $|\phi|(p)=0$ for every $\phi\in B^*$ annihilating $M$.  
In this section, we collect some general results on null projections. 
Before proceeding, we record an elementary inequality that will be used numerous times throughout the paper.

\begin{lemma}\label{L:stateineq}
Let $B$ be a $\rC^*$-algebra and let $\phi\in B^*$. Let $p\in B^{**}$ be a projection. Then, for every $x\in B^{**}$ we have
$
|\phi(px)|^2\leq \|\phi\| \|x\|^2|\phi|(p).
$
\end{lemma} 
\begin{proof} 
Choose a partial isometry $v\in B^{**}$ such that $\phi=|\phi|(\cdot v)$. Thus, by the Schwarz inequality we obtain  that 
\[
|\phi(px)|^2=||\phi|(pxv)|^2\leq \|\phi\||\phi|(pxv(pxv)^*)\leq  \|\phi\| 
\, \|x\|^2 \, |\phi|(p) 
\] 
for $x\in B^{**}$. \end{proof}

 Recall that a closed projection $p\in B^{**}$ is \emph{compact} if there is a contraction $b\in B$ such that $p=pb$. When $B$ is unital, this is clearly the same notion as that of being closed. 
We note that if $B$ is a nonunital $\rC^*$-algebra and  $A \subset B$ be a norm-closed subalgebra, then the 
generalized $A$-peak projections are the compact projections that happen to lie in $A^{\perp \perp}$.  See \cite{BR2013,BN2012} for this generalization of Glicksberg's theorem.
Thus in this case our ``Basic principle" stated in the introduction informs us that compact $A$-null projections are exactly the 
generalized $A$-peak projections which are also interpolation projections. 

The next development hinges on the following basic  fact.

\begin{lemma}\label{L:closeddense}
Let $B$ be a $\rC^*$-algebra. Then, the compact projections in $B^{**}$ span a weak-$*$ dense subspace.
\end{lemma}
\begin{proof}
Let $D\subset B^{**}$ denote the subspace generated by all compact projections in $B^{**}$. Let $\phi \in B^{*}$ be a functional annihilating $D$. We claim that $\phi=0$. To see this, it suffices to fix a positive element $b\in B$ and to show that $\phi(b)=0$. Let $\Omega\subset \bR$ denote the spectrum of $b$.
We recall that 
$\rC_0(\Omega\setminus\{0\})$ is $*$-isomorphic to the $\rC^*$-algebra generated by $b$. Then there exists a regular Borel measure $\mu$ on $\Omega \setminus \{ 0 \}$ such that
\begin{equation}\label{E:specmeasure}
\phi(f(b))=\int_\Omega f d\mu 
\end{equation}
for every $f\in \rC_0(\Omega)$ with $f(0)=0$.
If $K\subset \Omega \setminus \{ 0 \}$ is compact then $\chi_K(b) \in B^{**}$ by functional calculus, and $\chi_K(b)$ is closed with respect to the unitization of $B$, hence it
is a compact projection in $B^{**}$\cite[Lemma 2.4]{AAP1989}.  Hence, $\chi_K(b)\in D$ and $\phi(\chi_K(b))=0$.
On the other hand, by a dominated convergence  argument applied to a sequence of continuous functions decreasing to $\chi_K$ and vanishing at $0$, we find by virtue of \eqref{E:specmeasure} that $\phi(\chi_K(b))=\mu(K)$ so that $\mu(K)=0$. By regularity of the measure $\mu$, we conclude that $\mu=0$, and thus $\phi=0$.
\end{proof}

As motivation for what is to come, we give a simple criterion for the agreement between a $\rC^*$-algebra and one of its subspaces. For subalgebras of commutative $\rC^*$-algebras, one may use such 
considerations to  recover classical characterizations of $\rC(K)$ among its subalgebras, such as one due to Bade--Curtis \cite[Theorem B]{BC1966}.

\begin{proposition}\label{P:sacriterion}
Let $B$ be a $\rC^*$-algebra and let $M\subset B$ be a norm-closed subspace. Then, the following statements are equivalent.
\begin{enumerate}[{\rm (i)}]
\item $M=B$.
\item The unit  of $B^{**}$ is $M$-null.
\item Any compact projection in $B^{**}$ lies in $M^{\perp\perp}$.
\end{enumerate}
\end{proposition}
\begin{proof}
(i)$\Rightarrow$(ii): By assumption, we see that $M^{\perp}=\{0\}$, so that, vacuously, the unit of $B^{**}$ is $M$-null.

(ii) $\Rightarrow$(iii): Let $\phi\in M^\perp$ and $p\in B^{**}$ be a 
compact projection. Let $e\in B^{**}$ denote the unit. Then, $p\leq e$ so that $|\phi|(p)\leq |\phi|(e)=0$. This shows that $p$ is $M$-null, and in particular $p\in M^{\perp\perp}$ by Lemma \ref{L:stateineq}.

(iii)$\Rightarrow$(i): The assumption combined with Lemma \ref{L:closeddense} implies that $M^{\perp\perp}=B^{**}$. In turn, by the Hahn--Banach theorem, we find $M=B$ since $M$ is norm-closed.
\end{proof}

A natural question that arises from Proposition \ref{P:sacriterion} is whether there is a similar 
characterization of nullity for proper projections, as opposed to  the unit of $B^{**}$ in that result. 
As was communicated to us by Alexander Izzo, this is indeed the case for subalgebras of commutative $\rC^*$-algebras, at least for Borel projections -- this conversation with Izzo inspired several results in this section; see also the remark at the end of Section 6.

More precisely, let $X$ be a compact Hausdorff space and let $A\subset \rC(X)$ be a unital norm-closed subalgebra. Let $E\subset X$ be a Borel measurable subset, and denote by $\chi_E\in \rC(X)^{**}$ its characteristic function. Then, $\chi_E$ is $A$-null if and only if $\chi_K\in A^{\perp\perp}$ for every closed subset $K\subset E$. 
Our next goal is to establish such a characterization  in the non-commutative context.

\begin{proposition}\label{P:nullproj} 
Let $B$ be a $\rC^*$-algebra and let $M\subset B$ be a norm-closed subspace. Let $p\in B^{**}$ be a projection. Then, the following statements are equivalent. 
\begin{enumerate}[{\rm (i)}] 
\item $pB^{**}\subset M^{\perp\perp}$. 

\item For every compact projection $q\in B^{**}$, the element $pq$ lies in $M^{\perp\perp}$. 

\item The projection $p$ is $M$-null. 
\end{enumerate} 
\end{proposition} 
\begin{proof} 
(i) $\Rightarrow$ (ii): This is trivial. 

(ii) $\Rightarrow$ (iii): Fix $\phi\in M^\perp$. By assumption,  $\phi(pq)=0$ for every compact projection $q\in B^{**}$. By Lemma \ref{L:closeddense}, it follows that $\phi(px)=0$ for every $x\in B^{**}$. In particular, $|\phi|(p)=0$.

(iii) $\Rightarrow$ (i): Let $\phi\in M^\perp$. 
Then $|\phi|(p)=0$, so that $\phi(px)=0$ for each $x\in B^{**}$ by Lemma \ref{L:stateineq}.  We conclude that $pB^{**}\subset M^\perp$. 
\end{proof}

When $B$ is unital and commutative and $p$ is closed, statement (ii) above is equivalent to the stipulation that every closed projection dominated by $p$ lies in $M^{\perp\perp}$. Next, we give an alternate version of Proposition \ref{P:nullproj} that is more aligned with this formulation. For this purpose, we require some ideas from \cite{blecher2022}.

 Let $B$ be a unital $\rC^*$-algebra and let $q\in B^{**}$ be a closed projection. Denote by $B_q\subset B$ the $\rC^*$-subalgebra of elements in $B$ commuting with $q$. We say that a projection $r\in B_q^{\perp\perp}$ is \emph{closed with respect to $B_q$} if there is a net $(b_i)$ of positive contractions in $B_q$ decreasing to $r$ in the weak-$*$ topology of $B^{**}$. 
 Equivalently, using the weak-$*$ homeomorphic $*$-isomorphism $B_q^{**}\cong B_q^{\perp\perp}$,  this means that $r$ is closed when viewed as an element of $B_q^{**}$.  It is easily verified that $q$ itself lies in $B_q^{\perp\perp}$, and it is closed with respect to $B_q$ (see for instance \cite[Lemma 2.1]{blecher2022}).
 Indeed a projection $r\in B_q^{\perp\perp}$ is closed with respect to $B_q$ if and only if it is closed with respect to $B$. 
 
Next, let $A\subset B$ be a norm-closed subalgebra and put $A_q=A\cap B_q$. Then, $A_q$ may be viewed as a subalgebra of both $B^{**}$ and $B_q^{**}$.  Via the isomorphism $B_q^{**}\cong B_q^{\perp\perp}$, the double annihilator of $A_q$ inside of $B_q^{**}$ and its double annihilator inside of $B^{**}$ are identified with one another.  In what follows, we will reserve the notation $A_q^{\perp\perp}$ for the double annihilator of $A_q$ inside of $B^{**}$.

\begin{proposition} \label{P:commutanttrick} 
Let $B$ be a unital $\rC^*$-algebra, and let $A\subset B$ be a unital norm-closed subalgebra.  
Let $q\in B^{**}$ be a closed projection lying in $A^{\perp\perp}$. 
 Then, the following statements are equivalent.
\begin{enumerate}[{\rm (i)}]

\item Every closed subprojection $p$ of $q$ in 
$B_q^{\perp\perp}$ is a generalized $A_q$-peak projection (that is,
$p\in A_q^{\perp\perp}$).

\item $q$ is $A_q$-null with respect to $B_q$ 
(that is, $|\varphi | (q) = 0$ for all 
 $\varphi \in (A_q)^\perp$). 

\item $q$ is $A_q$-interpolating with respect to $B_q$ 
(that is, $q A_q = q B_q$). 
\end{enumerate} 
\end{proposition}
\begin{proof}  By Lemma 2.1 in  \cite{blecher2022} we automatically have $q \in A_q^{\perp\perp}$.
Then the  equivalence of  (ii) and (iii)  follows from the
 ``Basic principle" stated in the introduction.  
 Note that by definition of $B_q$ the projection $q$ commutes with the elements here. 

(i) $\Rightarrow$(ii):  
Let $\phi:B_q \to \bC$ be a continuous linear functional annihilating $A_q$.  Let $v\in B_q^{**}$ be a partial isometry such that $|\phi |=\phi(\cdot v)$. The goal is to show that $\phi(qv)=0$.  By virtue of Lemma \ref{L:closeddense}, it suffices to show that $\phi(qp)=0$ for every closed projection $p\in B_q^{**}$. Any such projection $p$ necessarily commutes with $q$ in $B_q^{**}$, so that $qp$ is closed in $B_q^{**}$ \cite[Theorem II.7]{akemann1969} and is dominated by $q$. In turn, viewing $qp$ as a closed subprojection of $q$ in $B_q^{\perp\perp}$, the assumption gives $qp\in A_q^{\perp\perp}$, so that as a projection in $B_q^{**}$, $qp$ lies in the double annihilator of $A_q$. In particular, $\phi(qp)=0$ as desired. 

(ii)$\Rightarrow$(i): Let $p\in B^{**}$ be any projection lying in $B_q^{\perp\perp}$ and dominated by $q$. Then, $p$ may be viewed as a projection in $B_q^{**}$ dominated by $q$. Our assumption then implies that $|\phi|(p)=0$ for every $\phi\in B_q^{*}$ annihilating $A_q$. In particular, by virtue of Lemma \ref{L:stateineq}, we infer that $\phi(p)=0$ for all such functionals $\phi$, and hence $p$ lies in the double annihilator of $A_q$ inside of $B_q^{**}$. In turn, when viewed as a projection in $B^{**}$, this means that $p$ lies in $A_q^{\perp\perp}$. 
 \end{proof}

\section{Measurably indistinguishable projections}\label{S:indist}
This short section contains a preliminary discussion on some non-commutative topological notions for  projections. These are used in the subsequent section.

We start with a motivating discussion. Given a topological space $X$ satisfying the $T_1$ separation axiom and distinct subsets $E,F\subset X$, there always exists an open subset $U\subset X$ containing $E$ but not containing $F$. In Akemann's non-commutative topology for projections, more diverse phenomena can arise, even in commutative contexts. 

\begin{example}\label{E:top}
Let $\bT\subset\bC$ denote the unit circle and let $B=\rC(\bT)$. For each $t\in \bT$, consider the closed projection $p_t=\chi_{\{t\}}\in B^{**}$. Let $\lambda\in B^*$ denote the state of integration against arc length measure. 
Let $e\in B^{**}$ be the projection satisfying
\[
 (1-e)B^{**}=\{x\in B^{**}: \lambda(x^*x)=0\}.
\]
Since $\{t\}$ has measure zero, we see that $\lambda(p_t)=0$ so that $p_t\leq 1-e$ for each $t\in \bT$, and thus $\bigvee_{t\in \bT}p_t\leq 1-e$. 
Fix a Borel measurable subset $E\subset \bT$ with non-zero measure. We clearly have $\bigvee_{t\in E}p_t\leq \chi_E$. On the other hand, $\chi_E(\lambda)>0$ so $\chi_E$ is not dominated by $1-e$, and in particular, $\chi_E\neq \bigvee_{t\in E}p_t$.

Nevertheless, the projections $\chi_E$ and $ \bigvee_{t\in E}p_t$ are nearly identical: their differences cannot be detected by means of certain Borel projections. More precisely, let $F\subset \bT$ be a Borel subset.  Assume that $\bigvee_{t\in E}p_t\leq \chi_F$, so that $\chi_{\{t\}}\leq \chi_F$ and $t\in F$ for each $t\in E$. This means that $E\subset F$, or $\chi_E\leq \chi_F$.  In other words, while $\bigvee_{t\in E}p_t$ and $\chi_E$ are distinct projections, they are dominated by precisely the same projections arising from Borel subsets.
\qed
\end{example}

Let $B$ be a 
$\rC^*$-algebra. By considering the usual universal representation of $B$, we find a Hilbert space $H_u$ such that $B^{**}\subset B(H_u)$. 
 As in  2.4.1 in  \cite{pedersen1979book} let $(B_{\text{sa}})^m\subset B^{**}$ denote the set of elements in $B^{**}$ that can be obtained as the increasing limit in the SOT of $B(H_u)$ of a self-adjoint net from $B$. A self-adjoint element $x\in B^{**}$ is said to be \emph{universally measurable} if, for every $\eps>0$ and state $\phi$ on $B$, there are elements $y,z\in(B_{\text{sa}})^m $ such that $-y\leq x\leq z$ and $\phi(y+z)<\eps$.  

Projections in $B^{**}$ that are either open or closed are universally measurable \cite[Proposition 3.11.9]{pedersen1979book}. More generally, the class of universally measurable projections also contains the Borel projections in the sense of \cite{combes1970}; see  \cite{brown2014} for more detail.

Let $p,q\in B^{**}$ be two projections. We say that $p$ and $q$ are \emph{measurably indistinguishable} if they are dominated by the same universally measurable projections in $B^{**}$. We will be concerned with the issue of measurable indistinguishability between a projection and the supremum of all minimal projections that it dominates. This is best framed using the following device.

Let $z_{\at}\in B^{**}$ denote the supremum of all minimal projections in $B^{**}$. It is well known that $z_{\at}$ is central: indeed, it coincides with the support projection for the unique weak-$*$ continuous extension to $B^{**}$ of the \emph{atomic representation} of $B$ \cite[Paragraph 4.3.7]{pedersen1979book}. In particular, by virtue of \cite[Lemma 4.3.8]{pedersen1979book}, we see that there is a collection $\{H_\lambda:\lambda\in \Lambda\}$ of Hilbert spaces such that $z_{\at}B^{**}\cong \prod_{\lambda\in \Lambda} B(H_\lambda)$. Consequently, if $p\in B^{**}$ is a projection, we have
\begin{equation}\label{Eq:atomic}
pz_{\at}=\bigvee\{ r: r \text{ minimal projection in } B^{**}, r\leq p\}.
\end{equation}
Henceforth, we will refer to $pz_{\at}$ as the \emph{atomic part} of $p$.  

\begin{proposition}\label{P:induniv}
Universally measurable projections in $B^{**}$ are measurably indistinguishable from their atomic parts.
\end{proposition} 
\begin{proof} 
Let $p,e\in B^{**}$ be universally measurable projections  with $e\geq pz_{\at}$. Since $z_{\at}$ is central, we find $(e-p)z_{\at}\geq 0$.  In turn, because $e-p$ is universally measurable \cite[Proposition 4.3.13]{pedersen1979book},  we  infer that $e\geq p$ by virtue of  \cite[Theorem 4.3.15]{pedersen1979book}.
 \end{proof}

Let us further analyze the case of closed projections, as this is what we will focus on in the next section.  The previous proposition can be interpreted as saying that, roughly speaking, a closed projection in $B^{**}$ essentially coincides with the supremum of all the minimal projections that it dominates. Of course, as illustrated in Example \ref{E:top}, this is not strictly true. Nevertheless, the following does hold.

\begin{corollary}\label{P:topindclosed}
Let $q\in B^{**}$ be a closed projection. Let $C\subset B^{**}$ denote the collection of all closed projections dominating the same minimal projections as $q$ does. Then, $q=\bigwedge_{r\in C} r$. 
\end{corollary} 
\begin{proof}
Since $z_{\at}$ is central, we have $q\geq qz_{\at}$, so  that $q\in C$ and $q\geq \bigwedge_{r\in C} r$. Conversely, fix $r\in C$ and an open projection $u\in B^{**}$ with $u\geq r$. Then, $u\geq qz_{\at}$ by \eqref{Eq:atomic}, so by Proposition \ref{P:induniv}, it follows that $u\geq q$. Since this holds for every such open projection $u$, applying \cite[Proposition 2.3]{hay2007}, we find that $q\leq r$. We conclude that $q\leq  \bigwedge_{r\in C} r$, so in fact $q=\bigwedge_{r\in C} r$ as desired. 
 \end{proof}

\section{The non-commutative Varopoulos phenomenon}\label{S:ncVar}

We begin by revisiting Example \ref{E:top} to illustrate that a closed projection may have a null atomic part without being null itself, despite both projections being measurably indistinguishable by Proposition \ref{P:induniv}.

\begin{example}\label{E:top2}
Let $\bT\subset\bC$ denote the unit circle, let $B=\rC(\bT)$. For each $t\in \bT$, consider the closed projection $p_t=\chi_{\{t\}}\in B^{**}$. Let $\lambda\in B^*$ denote the state of integration against arc length measure.   Let $K\subset \bT$ be a closed subset with positive measure. The atomic part of the closed  projection $\chi_K\in B^{**}$ is easily seen to be $\bigvee_{t\in K}p_t$ by \eqref{Eq:atomic}. 

Next, let $A\subset B$ denote the disc algebra, i.e.\ the norm-closed subalgebra generated by analytic polynomials.  It is well known that each point of $\bT$ is an $A$-peak point, so that $p_t$ is $A$-null for each $t\in \bT$. It follows that $\bigvee_{t\in E}p_t$ is also $A$-null. On the other hand, $\lambda(\chi_K)>0$, and hence $\chi_K$ is not an $A$-null \cite[Theorem 10.1.2]{rudin2008}
and not in $A^{\perp \perp}$, so that $K$ is not a peak set.  
\qed
\end{example}

The goal of this section is to exhibit conditions under which such a pathology cannot occur. This will be accomplished by proving a non-commutative analogue of Varopoulos' theorem \cite[Theorem 10.2.2]{rudin2008}. First, we record a topological preliminary.

\begin{lemma}\label{L:compact}
Let $B$ be a unital $\rC^*$-algebra and let $q\in B^{**}$ be a closed projection. Let $\{u_\lambda:\lambda\in \Lambda\}$ be a collection of open projections in $B^{**}$ that commute with $q$ and satisfy $q\leq \bigvee_{\lambda\in \Lambda}u_\lambda$. Then, there is a finite set $F\subset \Lambda$ such that $q\leq \bigvee_{\lambda\in F}u_\lambda$.
\end{lemma}
\begin{proof}
By assumption, we see that $q\wedge \left(\bigwedge_{\lambda\in \Lambda}(1-u_\lambda)\right)=0$. Invoking \cite[Proposition II.10]{akemann1969}, we find a finite subset $F\subset \Lambda$ such that $q\wedge \left( \bigwedge_{\lambda\in F}(1-u_\lambda)\right)=0$ or $q\wedge \left( 1-\bigvee_{\lambda\in F}u_\lambda\right)=0$. Equivalently, using the commutation assumption, we see that $q\leq \bigvee_{\lambda\in F}u_\lambda$ as desired.
\end{proof}

The previous result is stated to hold with no requirement of commutativity in \cite{hay2007}, but this  is not valid in general. 
Thankfully, the main claims of that paper appear to be unaffected by this.

We also need the following technical fact. Given a vector space $X$, we denote by $\bM_{m,n}(X)$ the space of $m\times n$ matrices with entries in $X$. When $m=n$, we simply write $\bM_n(X)$.

\begin{lemma}\label{L:duality}
Let $B$ be a $\rC^*$-algebra, let $q\in B^{**}$ be a closed projection and let $n\geq 1$. Let $R\in \bM_{1,n}(qB^{**})$ and $C\in \bM_{n,1}(B^{**}q)$. Fix $\Delta \in \bM_n(B)$, a functional $\phi\in B^*$ and a number $\eps>0$. Then, there are  $S \in \bM_{1,n}(qB)$, $T\in \bM_{n,1}(Bq)$ such that $\|S\|\leq \|R\|$,$ \|T\|\leq \|C\|$ and 
$
| \phi(R\Delta C-S\Delta T)|<\eps.
$
\end{lemma}
\begin{proof}
Let $J=B\cap (1-q)B^{**}$. Consider the map $\Phi:B/J\to qB$ defined as
\[
\Phi(b+J)=qb, \quad b \in B.
\]
By \cite[Proposition 3.1]{hay2007}, it follows that $\Phi$ is a completely isometric isomorphism. Hence, $\Phi^{**}:(B/J)^{**}\to (qB)^{**}$ is a completely isometric weak-$*$ homeomorphism. In turn, by \cite[Paragraph 1.4.4]{BLM2004}, there is a completely isometric  weak-$*$ homeomorphism $\Psi:(B/J)^{**}\to B^{**}/J^{\perp\perp}$ such that 
\[
\Psi(b+J)=b+J^{\perp\perp},\quad b\in B.
\]
Next, since $q$ is closed, 
and is the support projection of $J$ 
 \cite{akemann1969},
we obtain that $J^{\perp\perp}=(1-q)B^{**}$. Hence, the map $\Theta: B^{**}/J^{\perp\perp}\to qB^{**}$ defined as
\[
\Theta(x+J^{\perp\perp})=qx, \quad x\in B^{**}
\]
is also a completely isometric weak-$*$ homeomorphism. Therefore, we obtain a completely isometric weak-$*$ homeomorphism $\Omega=\Theta\circ \Psi \circ (\Phi^{**})^{-1} :(qB)^{**}\to qB^{**}$ acting as the inclusion on $qB$.

Invoke now \cite[Theorem 1.4.11]{BLM2004} to find an isometric isomorphism $$\Gamma:\bM_{1,n}(qB)^{**}\to \bM_{1,n}((qB)^{**})$$ acting as the inclusion on $\bM_{1,n}(qB)$. Then, the map $\Xi=\Omega^{(n)}\circ \Gamma: \bM_{1,n}(qB)^{**}\to \bM_{1,n}(qB^{**})$ is an   isometric weak-$*$ homeomorphism acting as the inclusion on  $\bM_{1,n}(qB)$.
Consider the weak-$*$ continuous linear functional $\psi$ on $\bM_{1,n}(qB)^{**}$ defined as
\[
\psi(X)= \phi(\Xi(X)\Delta C), \quad X\in \bM_{1,n}(qB)^{**}.
\]
By Goldstine's lemma
(namely that the unit ball of a Banach space is weak* dense in the ball of the bidual), we can find $S_0\in \bM_{1,n}(qB)$ with $\|S_0\|\leq \|\Xi^{-1}(R)\|= \|R\|$ such that
$$
|\psi(S_0-\Xi^{-1}(R))|<\eps/2.
$$
Put $S=\Xi(S_0)\in \bM_{1,n}(qB)$, which then satisfies $\|S\|\leq \|R\|$ and
\begin{equation}\label{Eq:gold1}
|\phi( S\Delta C-R\Delta C)|<\eps/2.
\end{equation}

To find $T$, we argue in a similar fashion. First, proceed as above to find  an isometric weak-$*$ homeomorphism $\Xi': \bM_{n,1}(Bq)^{**}\to \bM_{n,1}(B^{**}q)$ acting as the inclusion on  $\bM_{n,1}(Bq)$.  Consider  the weak-$*$ continuous linear functional $\psi'$ on $\bM_{n,1}(Bq)^{**}$ defined as
\[
\psi'(X)= \phi(S\Delta \Xi'(X)), \quad X\in \bM_{n,1}(Bq)^{**}.
\]
By Goldstine's lemma, we can find $T_0\in \bM_{1,n}(Bq)$ with $\|T_0\|\leq \|(\Xi')^{-1}C\|=\|C\|$ such that
\[
|\psi' (T_0- (\Xi')^{-1}(C))|<\eps/2.
\]
Put $T=\Xi'(T_0)\in \bM_{1,n}(Bq)$, which then satisfies $\|T\|\leq \|C\|$ and
\begin{equation}\label{Eq:gold2}
|\phi( S\Delta T-S\Delta C)|<\eps/2.
\end{equation}
Finally, combining \eqref{Eq:gold1} and \eqref{Eq:gold2}, we find
\[
| \phi(R\Delta C-S\Delta T)|<\eps
\]
as desired.
\end{proof}

Let $B$ be a $\rC^*$-algebra and let $A\subset B$ be a subalgebra. A projection $e\in B^{**}$ is said to be \emph{$A$-invariant} if $ae=eae$ for each $a\in A$. Further, $e$ is said to be \emph{$A$-coinvariant} if $1-e$ is $A$-invariant. We can now state and prove one of the main results of the paper.

 \begin{theorem}\label{T:ncVar}
 Let $B$ be a unital $\rC^*$-algebra and let $A\subset B$ be a unital norm-closed subalgebra.  Let $q\in B^{**}$ be a projection with the following properties:
 \begin{enumerate}[{\rm(i)}]
 \item $q$ is closed,
 \item there is $\kappa>0$ for which $q$ is a complete $(A,\kappa)$-interpolation projection,
 \item $q$ commutes with $A$, and
 \item the atomic part of $q$ lies in $A^{\perp\perp}$.
\end{enumerate}
Then, for every $\phi\in A^\perp$  and every open $A$-invariant projection $u\in B^{**}$ dominating $q$, we have
\[
 |\phi|(q)\leq  \kappa^3 \sqrt{|\phi|(1)\, | \phi|(u-q)}.
\]
 \end{theorem}
 \begin{proof}
 Let $0<\eps<1$ and let $\phi\in A^\perp$. We may assume without loss of generality that $\|\phi\|=1$, so that $|\phi|$ is a state. Throughout, we put $\tau=|\phi|$. There is a partial isometry $y\in B^{**}$ such that $\tau=\phi(\cdot y)$ and $\phi=\tau(\cdot y^*)$. By Goldstine's lemma, there is $t\in B$ such that $\|t\|\leq 1$ and  $|\phi(qy-qt)|<\eps.$
 Because $q$ is a complete $(A,\kappa)$-interpolation projection, there is $x\in A$ such that $qx=qt$ and $\|x\|\leq \kappa$.  In particular, we have
 \begin{equation}\label{Eq:pisom}
 \tau(q) =|\phi(qy)|< \eps+|\phi(qx)|.
 \end{equation}
Let $\{p_\lambda:\lambda\in \Lambda\}$ be the set of minimal projections in $B^{**}$ dominated by $q$, so that $q_{\at}=\bigvee_{\lambda\in \Lambda} p_\lambda$ is the atomic part of $q$ by \eqref{Eq:atomic}.
 It follows from \cite[Proposition II.4]{akemann1969} that each $p_\lambda$ is closed. Further, since $qA=qB$ we necessarily have $q_{\at}A=q_{\at}B$. Also, $q_{\at}\in A^{\perp\perp}$ by assumption, so the 
 ``Basic principle" from the introduction
  implies that $q_{\at}$ is $A$-null, and thus so is each $p_{\lambda}$. In particular, we have $p_\lambda\in A^{\perp\perp}$ for every $\lambda\in \Lambda$.

Let $u\in B^{**}$ be an open $A$-invariant projection dominating $q$. In particular, $u$ dominates the atomic part of $q$, so that $u\geq p_\lambda$ for every $\lambda$.  For each $\lambda$, we may thus apply \cite[Theorem 2.1]{BN2012} to find a contraction $a_\lambda\in A$ such that $a_\lambda p_\lambda=p_\lambda$ and 
\begin{equation}\label{Eq:norm1-u}
 \|(1-u)a_\lambda\|<\eps.
\end{equation}
Consider the open  projection $f_\lambda=\chi_{[0,\eps)}(|1-a_\lambda|)\in B^{**}$. Note that $q$ commutes with $A$, and hence with $f_\lambda$ as well.
Since $\|a_\lambda\|=1$ and $a_\lambda p_\lambda=p_\lambda$, we infer  $|1-a_\lambda| p_\lambda=0$ and thus 
\[
(1-f_\lambda)p_\lambda=\chi_{[\eps,1]}(|1-a_\lambda|)p_\lambda=0.
\] 
Therefore, $p_\lambda\leq f_\lambda$. 
It follows that the open projection $\bigvee_{\lambda}f_\lambda$ dominates $\bigvee_{\lambda}p_\lambda$, and hence it also dominates $q$ by Proposition \ref{P:induniv}. We may apply  Lemma \ref{L:compact} to find finitely many indices $\lambda_1,\ldots,\lambda_n$ such that $q\leq f_{\lambda_1}\vee \ldots \vee f_{\lambda_n}$.

 Put $e_1=f_{\lambda_1}$, and for $2\leq j\leq n$, let $e_j=\bigvee_{i=1}^j f_{\lambda_i}-\bigvee_{i=1}^{j-1} f_{\lambda_i}$.  It follows that the collection $\{e_1,\ldots,e_n\}$ is pairwise orthogonal and $\sum_{j=1}^n e_n =\bigvee_{j=1}^n f_{\lambda_j}$.
 Invoking \cite[Proposition V.1.6]{takesaki2002}, we find for each $2\leq j\leq n$ a partial isometry $v_j\in B^{**}$ such that $v_j^*v_j=e_j$ and $v_jv_j^*=e'_j$, where  
 $
  e_j'=f_{\lambda_j}-\left(\bigvee_{i=1}^{j-1}f_{\lambda_i} \right)\wedge f_{\lambda_j}.
 $
 For convenience, put $v_1=1$ and $e_1'=e_1$.
 We compute, for $1\leq j\leq n$, that
 \begin{align*}
 v_j^* a_{\lambda_j}v_j e_j -e_j&=v_j^* a_{\lambda_j}v_j v_j^*v_j -v_j^*v_jv_j^*v_j\\
  &=v_j^*(a_{\lambda_j}e'_j-e'_j)v_j
 \end{align*}
 whence
\[
  \|v_j^* a_{\lambda_j}v_j e_j -e_j\|\leq \|(a_{\lambda_j}-1)e_j'\|\leq \| |1-a_{\lambda_j}|f_{\lambda_j}\|
 \]
 since $e'_j\leq f_{\lambda_j}$.
 We conclude, for each $1\leq j\leq n$, that
 \begin{equation}\label{Eq:normej}
  \|v_j^* a_{\lambda_j}v_j e_j -e_j\|\leq \eps.
 \end{equation}
 Consider $V=\begin{bmatrix}qe_1v_1^* & \ldots & qe_n v_n^*\end{bmatrix}$ and $E=\begin{bmatrix} q e_1 & \ldots & qe_n \end{bmatrix}$ as elements in $\bM_{1,n}(qB^{**})$. We find 
 \[
 VV^*=\sum_{j=1}^n qe_j v_j^*v_j e_j q=qEE^*q
 \]
 and
 \[
 EE^*=\sum_{j=1}^n q e_j q=q \left(\bigvee_{j=1}^n f_{\lambda_j}\right)q=q
 \]
 whence $\|V\|=\|E\|=1$.
 If we set $\Delta=\diag(a_{\lambda_1},\ldots,a_{\lambda_n})$, then
 \begin{align*}
 V\Delta V^*-q&=\sum_{j=1}^n qe_j v_j^* a_{\lambda_j} v_j e_j q-\sum_{j=1}^n q e_j q\\
  &=\sum_{j=1}^n qe_j (v_j^* a_{\lambda_j} v_je_j-e_j) e_j q\\
  &=E\diag(v_j^*a_{\lambda_j}v_j e_j-e_j)E^*.
 \end{align*}
 By \eqref{Eq:normej} we infer
 \begin{equation}\label{Eq:rhoDelta}
 \|V\Delta V^*-q \|\leq \eps.
 \end{equation}
 By Lemma \ref{L:duality}, we  may now  choose $s_1,\ldots,s_n\in B$ and $t_1,\ldots,t_n\in B$ such that, upon setting 
 \[
  S=\begin{bmatrix}s_1& \ldots &s_n \end{bmatrix} \qand T=\begin{bmatrix}t_1\\ \vdots \\t_n \end{bmatrix} 
 \]
 we find 
 \[
 \|qS\|\leq \|V\|=1 \qand \|Tq\|\leq \|V^*\|=1
 \]
 along with
 \begin{equation}\label{Eq:sapprox}
 \left| \tau\left(( V \Delta V^*-qS\Delta Tq)xy^*\right)\right|<\eps.
 \end{equation}
 Since $q$ is a complete $(A,\kappa)$-interpolation projection, we can find $\alpha_1,\ldots,\alpha_n\in A$ along with $\beta_1,\ldots,\beta_n\in A$ such that, if we put
 \[
  R=\begin{bmatrix}\alpha_1& \ldots &\alpha_n \end{bmatrix} \qand C=\begin{bmatrix}\beta_1\\ \vdots \\ \beta_n \end{bmatrix},
 \]
 then 
 \begin{equation}\label{Eq:RCinterp}
 qS=qR \qand Tq=Cq
 \end{equation}
 while
 \[
 \|R\|\leq\kappa\|qS\|\leq\kappa \qand \|C\|\leq \kappa \|Tq\|\leq\kappa.
 \]
 We  define $d=R\Delta C=\sum_{i=1}^n \alpha_i a_{\lambda_i} \beta_i \in A$ and note that
 \begin{equation}\label{Eq:dnorm}
 \|d\|\leq \|R\| \|\Delta\| \|C\|\leq  \kappa^2.
 \end{equation} 
Since $u$ is $A$-invariant, we find $(1-u)d=(1-u)\sum_{j=1} \alpha_j (1-u)a_{\lambda_j} \beta_j$ so
\[
 \|(1-u)d\|\leq \|R\| \|C\| \|\diag((1-u)a_{\lambda_j})\|.
\]
Using \eqref{Eq:norm1-u}, we thus find
\begin{equation}\label{Eq:norm1-ud}
  \|(1-u)d\|\leq \kappa^2 \eps.
\end{equation}
On the other hand, using that $q$ commutes with $A$ we may write $qd=qdq$, and by \eqref{Eq:rhoDelta},\eqref{Eq:sapprox} and \eqref{Eq:RCinterp} we find that
\begin{align*}
|\tau( (qd-q)xy^*)|&=|\tau((qdq-q)xy^*)|=|\tau((qR\Delta Cq-q)xy^*)|\\
 &=|\tau((qS\Delta Tq-q)xy^*)|< \eps+|\tau((V\Delta V^*-q)xy^*)|\\
 &< (1+\|x\|)\eps
\end{align*}
whence
\begin{equation}\label{Eq:tauqdq}
|\tau( (qd-q)xy^*)|<(1+\kappa)\eps.
\end{equation}
Note that $dx\in A$ so that $\phi(dx)=0$ since $\phi\in A^\perp$. Therefore,
\begin{align*}
\phi(qx)&=\phi((q-d)x)=\tau((q-d)xy^*)\\
 &=\tau((q-qd)xy^*)-\tau((u-q)dxy^*)-\tau((1-u)dxy^*)
\end{align*}
so by Lemma \ref{L:stateineq} we find
\[
|\phi(qx)|\leq|\tau((qd-q)xy^*)|+ \|x\| \|d\| \tau(u-q)^{1/2}+\|x\|\|(1-u)d\|.
\]
Combined with \eqref{Eq:pisom},\eqref{Eq:dnorm},\eqref{Eq:norm1-ud} and \eqref{Eq:tauqdq}, this in turn implies
\begin{align*}
 \tau(q)&<\eps+|\phi(qx)|\leq \eps+ (1+\kappa)\eps+\kappa^3 \tau(u-q)^{1/2}+\kappa^3\eps.
\end{align*}
 Letting $\eps\to 0$ yields the desired conclusion.
 \end{proof}

 \subsection{Atomic parts and double annihilators}
Theorem \ref{T:ncVar} makes it advantageous to determine when the atomic part of a projection lies in the double annihilator of a subalgebra. As far as we know, there are no available general criteria guaranteeing this. We thus set out to identify a class of examples where this property is satisfied.

 Our class of examples will involve tensor products. We recall some relevant facts here, since several preliminaries are needed. 
Let $B$ and $D$ be $\rC^*$-algebras such that $B$ is nuclear. It follows from \cite[Theorem 3.1]{AB1980} that there is an injective $*$-homomorphism $\sigma: B^{**}\otimes_{\min} D^{**}\to (B\otimes_{\min} D)^{**}$ with the following properties.
\begin{enumerate}[{\rm (i)}]
\item $\sigma(b\otimes d)=b\otimes d$ for each $b\in B$ and $d\in D$.
\item Let $\sigma_B:B^{**}\to (B\otimes_{\min} D)^{**}$ and $\sigma_D:D^{**}\to (B\otimes_{\min} D)^{**}$ be the restriction $*$-homomorphisms arising from $\sigma$ (see for instance \cite[Theorem 3.2.6]{BO2008}). Then, $\sigma_B$ and $\sigma_D$ are both weak-$*$ continuous.
\end{enumerate}
Via this map $\sigma$, we will view $B^{**}\otimes_{\min} D^{**}$ as a $\rC^*$-subalgebra of $(B\otimes_{\min} D)^{**}$. In particular, given projections $q\in B^{**}$ and $r\in D^{**}$, $q\otimes r$ is a projection in $(B\otimes_{\min} D)^{**}$.  For the following, recall that if $A\subset B$ is a subalgebra, then $$A\otimes_{\min}D\subset B\otimes_{\min}D$$ by injectivity of the minimal tensor product.

\begin{lemma}\label{L:tensorAperpperp}
Let $B$ and $D$ be 
 $\rC^*$-algebras such that $B$ is nuclear. Let $q\in B^{**}$ and $r\in D^{**}$ be projections, with $r$ 
 open or closed. Let $A\subset B$ be a 
 closed subalgebra and assume that $q\in A^{\perp\perp}$. Then, $q\otimes r\in (A\otimes D)^{\perp\perp}$.
\end{lemma}
\begin{proof} 
We first assume that $r$ is open.  
Let $(d_j)$ be a net of positive contractions in $D$ increasing to $r$ in the weak-$*$ topology of $D^{**}$. Then $(\sigma_D(d_j))$ increases to $\sigma_D(r)$ in the weak-$*$ topology of $(B\otimes_{\min} D)^{**}$. 
In the universal representation of $B\otimes_{\min}D$, this implies that in fact $(\sigma_D(d_j))$ increases to $\sigma_D(r)$ in  the strong operator topology. On the other hand, 
by Goldstine's lemma there is a net $(a_i)$ of contractions in $A$ converging to $q$ in the weak-$*$ topology of $B^{**}$.   We see that the net $(\sigma_B(a_i)\sigma_D(d_j))_{(i,j)}$ converges to $\sigma(q\otimes r)$ in the weak-$*$ topology of $(B\otimes_{\min} D)^{**}$,
being  the product of a 
bounded weak-$*$ converging net and a bounded strongly converging net. Each $\sigma_B(a_i)\sigma_D(d_j)=\sigma(a_i\otimes d_j)$ lies in the image of $A\otimes D$ inside of $(B\otimes_{\min} D)^{**}$, and therefore $q\otimes r \in (A\otimes_{\min} D)^{\perp\perp}$.

If $r$ is closed, then by the above $q\otimes (1-r)$ lies in $(A\otimes_{\min} D)^{\perp\perp}$, and hence so does $q\otimes r$.
\end{proof}

Before proving the final preliminary result, recall that for a projection $p\in B^{**}$, the following statements are equivalent.
\begin{enumerate}[{\rm (i)}]
\item The projection $p$ is minimal in $B^{**}$.
\item $pB^{**}p=\bC p$.
\item 
There is a pure state $\omega$ on $B$ such that $\omega(b)p=pbp$ for each $b\in B$. 
\end{enumerate}
Using this equivalence, it is easily verified that if $p\in B^{**}$ and $q\in D^{**}$ are minimal projections, then the projection $p\otimes q$ is also minimal in $(B\otimes_{\min}D)^{**}$.

\begin{lemma}\label{L:minprojtensor}
Let $B$ and $D$ be $\rC^*$-algebras such that $B$ is commutative. Let $p\in (B\otimes_{\min} D)^{**}$ be a minimal projection. Then, there are minimal projections $e\in B^{**},f\in D^{**}$ such that $p=e\otimes f$.
\end{lemma}
\begin{proof}
There is a pure state $\omega$ on $B\otimes_{\min}D$ with support projection equal to $p$. By \cite[Theorem IV.4.14]{takesaki2002}, there are pure states $\phi$ of $B$ and $\psi$ of $D$ such that $\omega=\phi\otimes \psi$. Let $e\in B^{**}$ and $f\in D^{**}$ denote the support projections of $\phi$ and $\psi$, respectively. These are both minimal, so that $e\otimes f$ is minimal as well. On the other hand, we have $\omega(e\otimes f)=1$, so $e\otimes f \geq p$, which then forces $p=e\otimes f$. 
\end{proof}

We now arrive at our class of examples where  all projections have their atomic parts in the double annihilator of a subalgebra.

\begin{proposition}\label{P:CXD}
Let $X$ be a compact metric space and let $A\subset \rC(X)$ be a unital norm-closed subalgebra separating the points of $X$. Assume that $X$ is the Choquet boundary of $A$. Let $D$ be a unital $\rC^*$-algebra. Then, every minimal projection in $(\rC(X)\otimes_{\min}D)^{**}$ lies in $(A\otimes_{\min}D)^{\perp\perp}$. In particular, the atomic part of any projection in $(\rC(X)\otimes_{\min}D)^{**}$ lies in $(A\otimes_{\min}D)^{\perp\perp}$.
\end{proposition}
\begin{proof}
Let $q\in (\rC(X)\otimes_{\min}D)^{**}$ be a minimal projection. By Lemma \ref{L:minprojtensor}, there are minimal projections $e\in \rC(X)^{**}$ and  $f\in D^{**}$ such that $q=e\otimes f$. 
As is well-known, $e$ is simply the characteristic function of a point $x\in X$. By \cite[Theorem II.11.3]{gamelin1969}, 
we see that $x\in X$ is an $A$-peak point, since it lies in the Choquet boundary.  Hence $e\in A^{\perp\perp}$, so by Lemma \ref{L:tensorAperpperp}, we conclude that 
$q \in (A\otimes_{\min}D)^{\perp\perp}$. This proves the first statement. The second statement immediately follows from the first along with \eqref{Eq:atomic}.
\end{proof}

\subsection{Recovering Varopoulos' null projections}
We identify special cases where the conclusion of Theorem \ref{T:ncVar} can be strengthened to recover the fact that $q$ is $A$-null, as in Varopoulos' original result. More precisely, we impose a commutation condition between $A$ and the net approximating the closed projection $q$.

\begin{corollary}\label{C:ncVarcomm}
 Let $B$ be a unital $\rC^*$-algebra and let $A\subset B$ be a unital norm-closed subalgebra.  Let $q\in B^{**}$ be a projection with the following properties:
 \begin{enumerate}[{\rm(i)}]
 \item there is a net $(b_i)$ of positive contractions in $B$ commuting with $A$ and decreasing to $q$ in the weak-$*$ topology of $B^{**}$, 
 \item there is $\kappa>0$ for which $q$ is a complete $(A,\kappa)$-interpolation projection, and
 \item the atomic part of $q$ lies in $A^{\perp\perp}$.
\end{enumerate}
Then, $q$ is $A$-null.
\end{corollary}
\begin{proof}
First note that assumption (i) implies that $q$ is closed and commutes with $A$. Let $\eps>0$ and $\phi\in A^\perp$ with $\|\phi\|=1$. Let $C\subset B$ denote the $\rC^*$-subalgebra generated by $1$ and the elements $b_i$. Hence $A$ and $C$ commute. By virtue of \cite[Proposition II.3]{akemann1969} (and its proof), there is an open projection $u\in B^{**}$ lying in $C^{\perp\perp}$ such that $|\phi|(u-q)<\eps$. In particular, $u$ commutes with $A$, and hence is $A$-invariant.
Theorem \ref{T:ncVar} then gives $ |\phi|(q)\leq \kappa^3 \sqrt{\eps}$, and the proof is complete upon letting $\eps\to 0$.
\end{proof}
 
 Next, we show how the previous result implies Varopoulos' original statement.  Before proceeding with the proof, we record a useful fact about interpolation projections,  inspired by the trick used in \cite[Lemma 10.1]{DH2023}.

\begin{proposition}\label{P:interpnuclear}
Let $B$ be a unital nuclear $\rC^*$-algebra and let $A\subset B$ be a unital norm-closed subalgebra. Let $q\in B^{**}$ be a central closed projection such that $Aq=Bq$ as subspaces in $B^{**}$. Then, there is $\kappa>0$ such that $q$ is a complete $(A,\kappa)$-interpolation projection.
\end{proposition}
\begin{proof}
Consider the surjective completely contractive  map $\pi:A\to Bq$ defined as $\pi(a)=aq$ for each $a\in A$. Since $q$ is central, $Bq$ is a $\rC^*$-algebra and $\pi$ is a homomorphism. Therefore, there is a bounded homomorphism $\rho:Bq\to A/\ker \pi$ such that $\rho\circ \pi:A\to A/\ker \pi$ is simply the quotient map. Next, since $B$ is nuclear, 
so is $Bq$ \cite[Corollary 9.4.4]{BO2008}.  In addition, because $A/\ker \pi$ is an operator algebra (see \cite[Proposition 2.3.4]{BLM2004}), we may 
apply \cite[Theorem 4.1]{christensen1981} to see that $\rho$ is in fact completely bounded. This immediately implies the existence of $\kappa>0$ satisfying the required property. 
\end{proof}

We note in passing that in view of \cite{LT2012}, one-sided module maps on nuclear $\rC^*$-algebras do not have to be completely bounded.

We can now show how to recover Varopoulos' original theorem from our general machinery. 
 
 \begin{corollary}\label{C:Varopoulos}
 Let $X$ be a compact Hausdorff space and let $A\subset \rC(X)$ be a unital norm-closed subalgebra. Let $K\subset X$ be a subset with  the following properties:
\begin{enumerate}[{\rm (i)}]
\item $K$ is closed,
\item $A\chi_K=\rC(X)\chi_K$ as subspaces in $\rC(X)^{**}$ 
(that is, $K$ is an interpolation set), and 
\item $\chi_{\{x\}}\in A^{\perp\perp}$ for each $x\in K$
(that is, every point in $K$ is an $A$-peak point). 
\end{enumerate}
Then, $\chi_K$ is $A$-null (that is, $K$ is a peak set).
 \end{corollary}
 \begin{proof}
Let $q=\chi_K$, which is a closed projection in $\rC(X)^{**}$.  Invoking \eqref{Eq:atomic}, we see that the atomic part of $\chi_K$ is simply $\bigvee_{x\in K}\chi_{\{x\}}$. Given $x\in K$, we know that $\chi_{\{x\}}\in A^{\perp\perp}$. 
Hence, the atomic part of $\chi_K$ lies in $A^{\perp\perp}$. Next, it follows from Proposition \ref{P:interpnuclear} that there is $\kappa>0$ such that $q$ is a complete $(A,\kappa)$-interpolation projection. 
By virtue of Corollary \ref{C:ncVarcomm}, we conclude that $q$ is $A$-null.
 \end{proof}

\section{Application to algebras of analytic functions}\label{S:A(H)}
In this section, we illustrate how our general theorem can be used to recover some recent results pertaining to operator algebras of analytic functions on the unit ball \cite{DH2023}. This will be done through another special case of Theorem \ref{T:ncVar}, which we tackle first.

\subsection{Adapted projections}
Let $B$ be a $\rC^*$-algebra and let $A\subset B$ be a subalgebra. We say that a projection $e\in B^{**}$ is \emph{$A$-adapted} if there is a net $(e_i)$ of closed $A$-coinvariant 
 projections in $B^{**}$ increasing to $e$ in the weak-$*$ topology of $B^{**}$. The linchpin technical result of this section is the following refinement of Theorem \ref{T:ncVar}.

\begin{corollary}\label{C:adapted}
 Let $B$ be a unital $\rC^*$-algebra and let $A\subset B$ be a unital norm-closed subalgebra. Assume that there is a central closed projection $z\in B^{**}$ such that   $B^{**}z$ is commutative and $1-z$ is $A$-adapted. Let $q\in B^{**}$ be a projection with the following properties:
 \begin{enumerate}[{\rm(i)}]
 \item $q$ is closed,
 \item there is $\kappa>0$ for which $q$ is a complete $(A,\kappa)$-interpolation projection,
  \item $q\leq z$, 
 \item the atomic part of $q$ lies in $A^{\perp\perp}$.
\end{enumerate}
Then, $q$ is $A$-null.

\end{corollary}
\begin{proof}
 Let $\eps>0$ and let $\phi\in A^\perp$ with $\|\phi\|=1$. Let $\tau=|\phi|$, which is a state
 on $B$. By \cite[Proposition II.3]{akemann1969}, there is an open projection $u\in B^{**}$ dominating $q$ such that  $\tau(u-q)<\eps.$ Since $z$ is central and $q\leq z$, we find $q=qz\leq uz\leq u$ so that
 \begin{equation}\label{Eq:uq}
 \tau(uz-q)<\eps.
 \end{equation} 
 Further, because $1-z$ is $A$-adapted, there is an $A$-coinvariant closed projection $r\leq 1-z$ such that 
 \begin{equation}\label{Eq:1zr}
 \tau ((1-z)-r)<\eps.
 \end{equation} 
 Because $z$ is closed, so is $z(1-u)$  \cite[Proposition II.5]{akemann1969}. Moreover, we have $r\leq 1-z$ so that $r'=r+z(1-u)$ is closed by \cite[Theorem II.7]{akemann1969}. The projection \[u'=1-r'=((1-z)-r))+uz\]  is then seen to be open with $u'\geq uz\geq q$. 
We compute
 \begin{align*}
 \tau(u'-q)&=\tau(uz-q)+\tau((1-z)-r)\leq 2\eps
 \end{align*}
 by \eqref{Eq:uq} and \eqref{Eq:1zr}. 
 
 Observe next that since $B^{**}z$ is commutative, $q=qz$ is $A$-coinvariant and $uz$ is $A$-invariant.  On the other hand, $1-z$ is central and $r$ is $A$-coinvariant, so that $(1-z)-r$ is $A$-invariant. It follows that $u'$ is  $A$-invariant as well.  By virtue of Theorem \ref{T:ncVar}, we see that $$\tau(q)\leq \kappa^3\sqrt{\tau(u'-q)}\leq \kappa^3\sqrt{2\eps},$$ so the proof is complete upon letting $\eps\to 0$.
\end{proof}

\subsection{Hilbert function spaces and their multipliers}\label{SS:HFS}

Our goal is to show how Corollary \ref{C:adapted} can be used to recover \cite[Theorem 10.3]{DH2023}. 
Along the way, we will also recover a few more results from \cite{DH2023}.
Let us recall some background material and terminology; for more detail the reader may consult \cite{DH2023}
or e.g.\ \cite[Section 5]{CT2023Henkin}.

Let $d\geq 1$ be a positive integer and let $\bB_d\subset \bC^d$ denote the open unit ball. Let $\bS_d=\partial \bB_d$ denote the unit sphere. Let $\H$ be a Hilbert space of analytic functions on $\bB_d$ admitting  a reproducing kernel that is unitarily invariant and regular. Examples include the Dirichlet space on the disc, the Hardy and Bergman spaces, as well as the Drury--Arveson space. For each $\lambda\in \bB_d$, we let $k_\lambda\in \H$ denote the corresponding reproducing kernel vector. We then have that $\{k_\lambda:\lambda\in \bB_d\}$ spans a norm-dense subspace in $\H$ (see \cite{GHX04}).

A multiplier is a function $\phi:\bB_d\to\bC$ satisfying $\phi f\in \H$ for every $f\in \H$. The corresponding multiplication operator $M_\phi:\H\to\H$ is then bounded. The space of all multiplication operators induced by a multiplier is denoted by $\M(\H)$, and it is weak-$*$ closed subalgebra of $B(\H)$. If $\phi$ is a multiplier, then
\begin{equation}\label{Eq:evector}
M_\phi^* k_\lambda=\ol{\phi(\lambda)}k_\lambda, \quad  \lambda\in \bB_d.
\end{equation}
All polynomials are multipliers, and we let $\A(\H)\subset \M(\H)$ denote the norm-closure of these polynomial multipliers. A standard argument 
shows that if $M_\phi\in \A(\H)$, then $\phi$ is continuous on $\ol{\bB_d}$. 
 Indeed, for any multiplier $\phi$,  (\ref{Eq:evector}) shows that $\| \phi \|_\infty \leq \| M_\phi \|$; hence functions inducing the multiplication operators in $\A(\H)$ are uniform limits of polynomials on $\ol{\bB_d}$.

Denote by  $\T(\H)$ the $\rC^*$-subalgebra of $B(\H)$ generated by$\A(\H)$. It follows from \cite[Theorem 4.6]{GHX04} that $\T(\H)$ contains the ideal $K$ of compact operators  on $\H$. Let $z_K\in \T(\H)^{**}$ denote the closed central projection satisfying $\T(\H)^{**}(1-z_K)=K^{\perp\perp}$. This projection has the following crucial property.

\begin{lemma}\label{L:adapted}
The projection $1-z_K$ is $\A(\H)$-adapted.
\end{lemma}
\begin{proof}
For each finite subset $F\subset\bB_d$, let $e_F\in B(\H)$ denote the projection onto the subspace spanned by $\{k_\lambda:\lambda\in F\}$. Since the set $\{k_\lambda:\lambda\in \bB_d\}$ spans a norm-dense subset of $\H$, the net $(e_F)$ increases to the identity on $\H$ in the weak-$*$ topology of $B(\H)$. Further, each $e_F$ is compact on $\H$,  and it is $\A(\H)$-coinvariant by virtue of \eqref{Eq:evector}. Next, by \cite[Lemma 5.5]{CT2023Henkin}, there is a weak-$*$ homeomorphic $*$-isomorphism $\Omega: \T(\H)^{**}(1-z_K)\to B(\H)$ such that $\Omega(t(1-z_K))=t$ for every $t\in \T(\H)$. Since $e_F$ is compact, it lies in $\T(\H)$ and we see that $\Omega^{-1}(e_F)=e_F(1-z_K)=e_F$. It follows that $(\Omega^{-1}(e_F))$ is an increasing net of $\A(\H)$-coinvariant closed projections in $\T(\H)^{**}$ increasing to $1-z_K$ in the weak-$*$ topology of $\T(\H)^{**}$. 
\end{proof}

Let $\gamma:\T(\H)\to \T(\H)/K$ denote the quotient map.  There is a natural $*$-isomorphism $\tau:\T(\H)/K\to \rC(\bS_d)$ such that
\[
(\tau\circ\gamma)(M_\phi)=\phi,\quad M_\phi\in \A(\H).
\]
See e.g.\ \cite{GHX04}. In particular, $\T(\H)$ is nuclear \cite[Exercise 3.8.1]{BO2008}. The restriction of $(\tau\circ \gamma)^{**}$ gives a weak-$*$ homeomorphic $*$-isomorphism $\Theta: \T(\H)^{**}z_K\to  \rC(\bS_d)^{**}$. We now give a description of the closed projections in $\T(\H)^{**}$ dominated by $z_K$.

\begin{lemma}\label{L:closedsetproj}
Let $q\in \T(\H)^{**}$ be a projection. Then, the following statements hold.
\begin{enumerate}[{\rm (i)}]
\item 
 If $q$ is closed, then there is a closed subset $E\subset \bS_d$ such that $(\tau\circ \gamma)^{**}(q)=\chi_E$.
\item Assume that $q\leq z_K$. If  there is a closed subset $E\subset \bS_d$ such that $\Theta(q)=\chi_E$, then $q$ is closed.
\end{enumerate}
\end{lemma}
\begin{proof}
(i) Let $(t_i)$ be a net of positive contractions in $\T(\H)$ decreasing to $q$ in the weak-$*$ topology of $\T(\H)^{**}$. Since $\Theta$ is a weak-$*$ continuous $*$-homomorphism 
with $\Theta(z_K) = 1$,  
we have that $\Theta(t_i z_K) = (\tau\circ \gamma)^{**}(t_i) =( \tau\circ \gamma)(t_i) \in \rC(\bS_d)$.
It follows that $(\Theta(t_i z_K) )$ is a  net of positive contractions in $\rC(\bS_d)$ decreasing to $\Theta(q)$ in the weak-$*$ topology of $\rC(\bS_d)^{**}$. In other words, $\Theta(q)$ is closed in $\rC(\bS_d)^{**}$, so  there is a closed subset $E\subset \bS_d$ such that $\Theta(q)=\chi_E$.

(ii) Since $1-\chi_E$ is open in $\rC(\bS_d)^{**}$, there is a closed two-sided ideal $J\subset \rC(\bS_d)$ such that $J^{\perp\perp}=\rC(\bS_d)^{**}(1-\chi_E)$. Let $\pi:\rC(\bS_d)\to \rC(\bS_d)/J$ denote the corresponding quotient map. 
Then,
\begin{align*}
\ker(\pi\circ \tau\circ \gamma)^{\perp\perp}&=\ker(\pi\circ \tau\circ \gamma)^{**}=\T(\H)^{**}(1-z_K)+\T(\H)^{**}(z_K-q)\\
&=\T(\H)^{**}(1-q).
\end{align*}
We conclude that $q$  is closed in $\T(\H)^{**}$. 
\end{proof}

A closed set $E\subset \bS_d$ is an \emph{$\A(\H)$-interpolation set} if 
\[
\chi_E \, \{\phi:M_\phi\in \A(\H)\} =\chi_E \, \rC(\bS_d) 
\]
as subsets of $\rC(\bS_d)^{**}$. This is related to our notion of interpolation projection in $\T(\H)^{**}$, as follows.

\begin{lemma}\label{L:interpset}
Let $q\in \T(\H)^{**}$ be a closed projection dominated by $z_K$. Then, the following statements are equivalent.
\begin{enumerate}[{\rm (i)}]
\item 
$q\A(\H)=q\T(\H)$.
\item There is a closed $\A(\H)$-interpolation subset $E\subset \bS_d$ such that $\Theta(q)=\chi_E$.
\end{enumerate}
\end{lemma}
\begin{proof}
By Lemma \ref{L:closedsetproj}, there is a closed  subset $E\subset \bS_d$ such that $\Theta(q)=\chi_E$.
Since $\Theta$ is a $*$-isomorphism such that
\[
\Theta(M_\phi z_K)=\phi, \quad M_\phi\in \A(\H),
\]
we infer that $q\A(\H)=q\T(\H)$ is equivalent to $\Theta(q\A(\H))=\Theta(q\T(\H))$, or
\[
\chi_E \, \{\phi:M_\phi\in \A(\H)\} =\chi_E \, \rC(\bS_d) .\qedhere
\] 
\end{proof}

We remark here that for a projection $q\leq z_K$, the equality in statement (i) above is equivalent to $\A(\H)q=\T(\H)q$. Indeed, in this case $q$ is central since $\T(\H)^{**}z_K\cong \rC(\bS_d)^{**}$.

A bounded linear functional $\phi:\A(\H)\to\bC$ will be said to be \emph{$\M(\H)$-Henkin} if $\phi$ extends weak-$*$ continuously to $\M(\H)$. 
We note here that this coincides with the notion used in \cite{CT2023Henkin} 
by Equation (6) above Lemma 5.6 there
(or  \cite[Lemma 2.1]{DH2023}). This will be used implicitly below, whenever results from \cite{CT2023Henkin} regarding Henkin functionals are invoked.

A regular Borel measure $\mu$ on $\bS_d$ is said to be $\M(\H)$-Henkin if $I_\mu\circ \tau\circ \gamma|_{\A(\H)}$ is $\M(\H)$--Henkin in the previous sense, where $I_\mu:\rC(\bS_d)\to\bC$ is the integration functional defined as
\[
I_\mu(f)= \int_{\bS_d}fd\mu, \quad f\in \rC(\bS_d).
\]

A closed projection $q\in \T(\H)^{**}$ is \emph{$\M(\H)$-totally null} if $|\phi|(q)=0$ for every functional $\phi\in \T(\H)^*$ such that $\phi|_{\A(\H)}$ is $\M(\H)$-Henkin. A Borel subset $E\subset \bS_d$ is said to be $\M(\H)$-totally null if $|\mu|(E)=0$ for every
 $\M(\H)$-Henkin measure. We clarify the relation between these notions in the next result. In what follows, we will repeatedly use the fact that $\phi(\cdot (1-z_K))|_{\A(\H)}$ is an $\M(\H)$-Henkin functional for every $\phi\in \T(\H)^*$ (see \cite[Lemma 5.5]{CT2023Henkin}).

\begin{theorem}\label{T:nullAH} Let $\H$ be a Hilbert space of analytic functions on $\bB_d$ admitting  a reproducing kernel that is unitarily invariant and regular. 
Let $q\in \T(\H)^{**}$ be a projection. Then, the following statements are equivalent.
\begin{enumerate}[{\rm (i)}]
\item The projection $q$ is closed and  $\A(\H)$-null.
\item The projection $q$ is closed and $\M(\H)$-totally null.
\item The projection $q$ is dominated by $z_K$, and there is a closed $\M(\H)$-totally null subset $E\subset \bS_d$ such that $\Theta(q)=\chi_E$.\
\end{enumerate}
\end{theorem}
\begin{proof}
(i) $\Leftrightarrow$ (ii): By
Lemma 5.6 in  \cite{CT2023Henkin} and Equation (6) above it, we see that \cite[Theorem 6.2]{CT2023Henkin} may be applied to the triple $\A(\H)\subset \T(\H)\subset B(\H)$, which in turn gives the desired equivalence.

(ii) $\Rightarrow$ (iii): Let $\phi\in \T(\H)$ be a state. Put $\psi=\phi(\cdot (1-z_K))$, which is positive since $z_K$ is central. Using that $\psi|_{\A(\H)}$ is an $\M(\H)$-Henkin functional, we find $\psi(q)=0$, or $\phi(q(1-z_K))=0$. Since $\phi$ was arbitrary, we infer that $q(1-z_K)=0$ or $q\leq z_K$. 

By Lemma \ref{L:closedsetproj}, there is a closed set $E\subset\bS_d$ such that $\chi_E=\Theta(q)$. We claim that $E$ is $\M(\H)$-totally null. To see this, let $\mu$ be an $\M(\H)$-Henkin measure.  Then, $|I_\mu\circ \tau\circ \gamma|(q)=0$.  By \cite[Lemma 2.3]{CT2023Henkin} we find $|I_\mu\circ \tau\circ \gamma|=|I_\mu| \circ \tau^{**}\circ \gamma^{**}$ so
\begin{align*}
|\mu|(E)&=|I_\mu|(\chi_E)=(|I_\mu| \circ \Theta)(q)=(|I_\mu| \circ \tau^{**}\circ \gamma^{**})(q)=|I_\mu\circ \tau\circ \gamma|(q)=0
\end{align*}
which establishes the claim.

(iii) $\Rightarrow$ (ii): The fact that $q$ is closed follows from Lemma \ref{L:closedsetproj}. Let $\phi\in \T(\H)^*$ be such that $\phi|_{\A(\H)}$ is an $\M(\H)$-Henkin functional. The goal is to show that $|\phi|(q)=0$. By assumption, we have $q\leq z_K$, so it is equivalent to show that $|\phi|(qz_K)=0$. Using \cite[Theorem 5.7]{CT2023Henkin}, it follows that the restriction to $\A(\H)$ of $\psi=|\phi|(\cdot z_K)$ is also an $\M(\H)$-Henkin functional. The functional $\psi \circ \Theta^{-1}$ is bounded and positive on $\rC(\bS_d)$, so there  is a regular positive Borel measure $\mu$ on $\bS_d$ such that $I_\mu=(\psi\circ \Theta^{-1})|_{\rC(\bS_d)}$. By construction, $\mu$ is an $\M(\H)$-Henkin measure, so that $\mu(E)=0$.

Since $q$ is closed and $\T(\H)$ is 
finitely generated by the monomials, hence is separable, we may choose  a sequence of positive contractions $(t_n)$ in $\T(\H)$ decreasing to $q$ in the weak-$*$ topology of $\T(\H)^{**}$. In particular, we see that the bounded sequence of continuous functions $((\tau\circ \gamma)(t_n))$ converges pointwise on $\bS_d$ to $\chi_E$. By dominated convergence and using that 
$(\tau \circ \nu)^{**}(z_K) = 1$, we then find
\begin{align*}
\psi (q)&=\lim_{n\to\infty} \psi(t_n z_K)=
\lim_{n\to\infty}(\psi \circ \Theta^{-1})( (\tau\circ \gamma)(t_n)))\\
&=\lim_{n\to\infty}\int_{\bS_d}(\tau\circ \gamma)(t_n)d\mu=\int_{\bS_d}\chi_E d\mu=\mu(E)=0.
\end{align*}
We conclude that $q$ is indeed $\M(\H)$-totally null.
\end{proof}

\subsection{The Davidson--Hartz theorems}
The previous result has bearing on the topic of peak-interpolation. Indeed, it allows us to provide alternative proofs (and even some refinements) of some of the results from \cite{DH2023} via the noncommutative peak-interpolation theory surveyed in \cite{blecher 2013}. For the most part, the statement below is a commutative one. That our approach takes us through the abstract noncommutative theory is a novel feature that we feel has potential for further future applications.


%

\begin{theorem} \label{T:DHbish} Let $\H$ be a Hilbert space of analytic functions on $\bB_d$ admitting  a reproducing kernel that is unitarily invariant and regular.  Let $E\subset \bS_d$ be a closed $\M(\H)$-totally null set. Then, for every non-zero function $g\in \rC(E)$ and every strictly positive function $k\in \rC(\bS_d)$ such that $|g|\leq k$ on $E$, there is $M_\phi\in \A(\H)$ with the following properties:
\begin{enumerate}[{\rm (i)}]
\item $\|M_\phi\|=\|g\|$,
\item $\phi|_E=g$,
\item $|\phi|\leq k$ on $\bS_d$,
and
\item $|\phi|<\|g\|$ on $\bB_d$.
\end{enumerate}
\end{theorem}
\begin{proof}  
Let $g\in \rC(E)$ with $\|g\|=1$. Let $k_0\in \rC(\bS_d)$ be a strictly positive function such that  $k_0|_E=1$ and $0<k<1$ on $\bS_d\setminus E$. Let $k \in \rC(\bS_d)$ be strictly positive such that $|g| \leq k$ on $E$ and put $k_1=\min\{k,k_0\}$.  By Tietze's extension theorem, there is $g_0\in \rC(\bS_d)$ such that $g_0|_E=g$ and $\|g_0\|=1$. Recall that $\tau\circ \gamma:\T(\H)\to \rC(\bS_d)$ is a surjective $*$-homomorphism. Thus,  there is $b\in \T(\H)$ with $\|b\|=1$ and a positive element $d\in \T(\H)$ with  $\|d\|=\|k_1\|$ such that $(\tau\circ \gamma)(b)=g_0$ and and $(\tau\circ \gamma)(d)=k_1$.

Next, by Lemma \ref{L:closedsetproj}
 there is a closed projection $q \leq z_K$ such that $\Theta(q) = \chi_E$.  Note that $q$ is central in $\T(\H)^{**}$ since $\T(\H)^{**} z_K$ is abelian. 
 Moreover, $q$ is $\A(\H)$-null by Theorem \ref{T:nullAH}. By construction, we have
 \begin{align*}
 \Theta(d^2q)=k_1^2 \chi_E\geq |g_0|^2\chi_E=\Theta(q b^* b q)
 \end{align*}
so that $q b^* b q \leq d^2q$.
Invoking now \cite[Theorem 3.4]{blecher2013} or 
\cite[Theorem 2.3 (iii)]{blecher2022}, we obtain $a \in \A(\H)$
with $aq = qa = bq$ and $a^* a \leq d^2$.  There is a multiplier $\phi$ such that $a=M_\phi$. We thus obtain
\[
\phi\chi_E=\Theta(aq)=\Theta(bq)=g_0 \chi_E
\]
whence $\phi|_E=g$ and (ii) holds. Since $M_\phi^*M_\phi\leq d^2$ we find $\|M_\phi\| \leq \|d\|= \|k_1\| \leq 1=\|g\|$.  This  in turn this forces $\|M_\phi\|=\|g\|$, since $$\| M_\phi \| \geq \| aq \| = \| bq \| = \| g_0 \chi_E \| = \| g \|.$$
Thus we have  (i). 
In addition, applying $\Theta$ to  the inequality $M_\phi^*M_\phi\leq d^2$ implies $|\phi|^2 \leq |k_1|^2$ on $\bS_d$, and hence (iii) is satisfied. Finally, because $|\phi|\leq  k_1\leq k_0$, we infer that $\phi$ is not constant on $\bS_d$, and hence $|\phi|<1$ on $\bB_d$ by the maximum modulus principle.
\end{proof} 

The previous result is a sharpening of the scalar version of  \cite[Theorem 8.1]{DH2023}.  
Indeed, it shows that peak-interpolation can be implemented  while maintaining domination by a control function, in line with Bishop's original theorem \cite{bishop1962}. For this reason, let us say that a closed subset $E\subset \bS_d$ is a \emph{Bishop peak-interpolation set} if the conclusion of Theorem \ref{T:DHbish} holds,
that is every non-zero function $g \in C(E)$ may be `interpolated' as in (i)--(iv) of the theorem. 
By Theorem \ref{T:DHbish} this class of sets coincides with what are called {\em peak interpolation} sets in \cite{DH2023}.  One direction of this follows  from (i)--(iv) with $k = \| g \| \, k_0$, where $k_0$ is defined in the first line of the proof above.  Conversely,  any peak set in the sense of \cite{DH2023} satisfies Theorem \ref{T:DHbish} since it is $\M(\H)$-totally null  by the standard argument  in the proof of 
 \cite[Theorem 12.2]{DH2023}.   
 So we see that all these classes, and  the peak sets, and the $\M(\H)$-totally null sets,  coincide.

Next, we recover  \cite[Theorem 10.3]{DH2023}. 

\begin{theorem}\label{T:DH}Let $\H$ be a Hilbert space of analytic functions on $\bB_d$ admitting  a reproducing kernel that is unitarily invariant and regular. 
Let $E\subset \bS_d$ be a closed $\A(\H)$-interpolation set. 
Assume that there exists a non-empty $\M(\H)$-totally null set. Then, $E$ is $\M(\H)$-totally null. 
\end{theorem} 
\begin{proof} 
Consider the projection $q=\Theta^{-1}(\chi_E)\in \T(\H)^{**}z_K$, which is closed by Lemma \ref{L:closedsetproj}. To show that $E$ is $\M(\H)$-totally null, it is equivalent to show that $q$ is $\A(\H)$-null, by Theorem \ref{T:nullAH}. Note that $1-z_K$ is $\A(\H)$-adapted by Lemma \ref{L:adapted}, and $q\leq z_K$. Thus, the fact that $q$ is $\A(\H)$-null will follow from Corollary \ref{C:adapted}, once we verify all the required conditions.

Using that $E$ is an $\A(\H)$-interpolation set, we find 
 $q\A(\H)=q\T(\H)$  by Lemma \ref{L:interpset}.
Recall that $\T(\H)^{**}z_K \cong \rC(\mathbb{S}_d)^{**}$ and $q\leq z_K$, so that $q$ is central.
Since $\T(\H)$ is nuclear, by Proposition \ref{P:interpnuclear} we infer that there is $\kappa>0$ such that $q$ is a complete $(\A(\H),\kappa)$-interpolation projection.

It only remains to show that the atomic part of $q$ lies in $\A(\H)^{\perp\perp}$. Invoking \eqref{Eq:atomic}, we see that the atomic part of $\Theta(q)$ is simply $\bigvee_{\zeta\in E} \chi_{\{\zeta\}}$, so that the atomic part of $q$ is $\bigvee_{\zeta\in E} p_{\{\zeta\}}$, with $p_\zeta=\Theta^{-1}(\chi_{\{\zeta\}})$. Now,  \cite[Lemma 9.1]{DH2023} implies that the set $\{\zeta\}$ is $\M(\H)$-totally null set, whence $p_\zeta$ is $\A(\H)$-null by Theorem \ref{T:nullAH}, and in particular $p_\zeta\in \A(\H)^{\perp\perp}$.
\end{proof}

The following result closes the circle of results about peak and null sets and projections for $\A(\H)$. Recall that by \cite[Lemma 5.5]{CT2023Henkin}, there is a weak-$*$ homeomorphic $*$-isomorphism $\Omega: \T(\H)^{**}(1-z_K)\to B(\H)$ such that $\Omega(t(1-z_K))=t$ for every $t\in \T(\H)$.

\begin{corollary}\label{C:DHequiv}
Let $\H$ be a Hilbert space of analytic functions on $\bB_d$ admitting  a reproducing kernel that is unitarily invariant and regular. Let $q\in \T(\H)^{**}$ be a closed projection. Then, the following statements are equivalent. 
\begin{enumerate}[{\rm (i)}]
\item $q$ is an $\A(\H)$-peak projection.
\item $q$ is $\A(\H)$-null.
\item $q$ is dominated by $z_K$ and there is a closed $\M(\H)$-totally null subset $E\subset \bS_d$ such that $\Theta(q)=\chi_E$.
\item $q$ is dominated by $z_K$ and there is a closed Bishop peak-interpolation subset $E\subset \bS_d$ such that $\Theta(q)=\chi_E$.
\end{enumerate} 
If there exist non-empty $\M(\H)$-totally null sets, then these statements are further equivalent to the following.
\begin{enumerate}
\item[{\rm (v)}]  $q$ is dominated by $z_K$ and there is a closed $\A(\H)$-interpolation set $E\subset \bS_d$ such that $\Theta(q)=\chi_E$.
\end{enumerate}
\end{corollary}
\begin{proof}
(ii) $\Leftrightarrow$ (iii): This follows from Theorem \ref{T:nullAH}.

(ii) $\Rightarrow$ (i): This follows from the ``Basic principle" 
in the introduction.

(i) $\Rightarrow$ (iii):  By assumption, there is a contraction $a \in \A(\H)$ such that $aq=q$ and $\|ap\|<1$ for every closed projection $p\in \T(\H)^{**}$ with $pq=0$. Consider the projection $P=\Omega(q(1-z_K))$ in $B(\H)$. Since the sequence $(a^n)$ converges to $q$ in the weak-$*$ topology of $\T(\H)^{**}$ \cite[Lemma 3.6]{hay2007}, it follows that the sequence 
$\Omega(a^n (1-z_K))=a^n$ converges to $P$ in the weak-$*$ topology of $B(\H)$. Then, any vector $f\in \H$ in the range of $P$ satisfies $af=f$. Since $\H$ consists of analytic functions on the ball, by the identity theorem we conclude that $a=1$ if $P\neq 0$. But $a\neq 1$ since $q$ is $\A(\H)$-null. Therefore, $P=0$, and we conclude that $(a^n)$ converges to $0$ in the weak-$*$ topology of $B(\H)$. 

Next, let $\phi\in \T(\H)^*$ be an $\M(\H)$-Henkin functional. By virtue of \cite[Theorem 5.7]{CT2023Henkin}, $|\phi|$ is also $\M(\H)$-Henkin. By definition, this implies that
\[
|\phi|(q)=\lim_{n\to\infty}|\phi|(a^n)=0
\]
so $q$ is $\M(\H)$-totally null. We see that (iii) follows from Theorem \ref{T:nullAH}.

(iii) $\Rightarrow$ (iv): This is Theorem \ref{T:DHbish}.

(iv) $\Rightarrow$ (iii):  See the remark after Theorem \ref{T:DHbish}.

It is trivial that (iv) $\Rightarrow$ (v), while (v)$\Rightarrow$ (iii) follows from Theorem \ref{T:DH} under the assumption that there exist non-empty $\M(\H)$-totally null sets.
\end{proof}

We 
record some consequences of our arguments so far, using also the facts established in the remark after Theorem \ref{T:DHbish}.  
We reiterate that we are recovering results from \cite{DH2023} here using our techniques. 

\begin{corollary}\label{bij} Let $\H$ be a Hilbert space of analytic functions on $\bB_d$ admitting  a reproducing kernel that is unitarily invariant and regular.  
The $*$-isomorphism  $\Theta:\T(\H)^{**}z_K\to \rC(\bS_d)^{**}$ restricts to an order preserving  bijection  between the 
class of closed projections $q$ characterized in Corollary \ref{C:DHequiv} and the class of Bishop peak-interpolation (or equivalently,  peak, or $\M(\H)$-totally null) subsets $E$ of $\bS_d$.
\end{corollary}

Using this bijection one may 
add `interpolation set' to the equivalences in the last line, if nonempty peak sets exist.
To close this section,
we relate the statements above to the condition that 
$q\A(\H)=q\T(\H)$. The key observation is the following. We say that a vector $\xi\in \H$ is 
\emph{strictly $*$-cyclic} for  $\A(\H)$ if $\A(\H)^*\xi=\H$. 


\begin{lemma}\label{L:qz}
Let $q\in \T(\H)^{**}$ be a closed projection such that $q\A(\H)=q\T(\H)$. If $\A(\H)$ admits no strictly $*$-cyclic vector in $\H$, then $q\leq z_K$.
\end{lemma}
\begin{proof}
If we put $P=\Omega(q(1-z_K))$ as before, then it suffices to show that $P=0$. As subspaces in $\T(\H)^{**}$, we have $\A(\H)^*q=\T(\H)q$, which in turn implies $\A(\H)^*q(1-z_K)=\T(\H)q(1-z_K)$. Upon applying $\Omega$, we find $\A(\H)^*P=\T(\H)P$ as subsets in $B(\H)$. If $P\neq 0$, then there is a unit vector $\xi\in \H$ in the range of $P$; let $E\in B(\H)$ denote the corresponding rank-one projection. Since $E\leq P$, we have $\A(\H)^*E=\T(\H)E$. Let $\eta\in \H$ be an arbitrary vector. Choose a rank-one operator  $t\in B(\H)$ such that $t\xi=\eta$. Then, $t\in 
K(\H) \subset \T(\H)$ so there is $a\in \A(\H)$ with $a^*E=tE$. This implies that 
$
a^*\xi=a^*E\xi=tE\xi=\eta.
$
In other words, $\A(\H)^*\xi=\H$, so that $\xi$ is strictly $*$-cyclic for $\A(\H)$.
\end{proof}

We can now give a complement to Corollary \ref{C:DHequiv}.

\begin{theorem}\label{T:ncDH}
Let $q\in \T(\H)^{**}$ be a closed projection such that $q\A(\H)=q\T(\H)$. Assume that
\begin{enumerate}[{\rm (i)}]
\item $\A(\H)$ admits no strictly $*$-cyclic vector in $\H$, and
\item there exists a non-empty $\M(\H)$--totally null set.
\end{enumerate}
Then, $q$ is $\M(\H)$-totally null, and hence is $\A(\H)$-null.
\end{theorem}
\begin{proof}
By Lemma \ref{L:qz}, we see that $q\leq z_K$. Hence, by Lemma  \ref{L:interpset}, there is a closed $\A(\H)$-interpolation subset $E\subset \bS_d$ such that $\Theta(q)=\chi_E$.  By virtue of Theorem \ref{T:nullAH}, to show that $q$ is $\M(\H)$-totally null, it is then equivalent to show that $E$ is $\M(\H)$-totally null, which follows at once from Theorem \ref{T:DH}.
\end{proof}

We do not know of a characterization of those spaces $\H$ for which $\A(\H)$ admits no strictly $*$-cyclic vector. Nevertheless, we can give a simple sufficient condition: such vectors cannot exist provided that $\A(\H)$ is properly contained in $\M(\H)$ (which is the case in many classical examples). 
Indeed, this follows from the following general fact. 

If $A\subset B(H)$ is a commutative subalgebra \emph{properly} contained in its commutant $A'$, then $A$ admits no strictly $*$-cyclic vector. To see this, assume that $\xi\in H$ is strictly $*$-cyclic for $A$, and choose $b\in A'$ with $b\notin A$. Then, there is $a\in A$ such that $a^*\xi=b^*\xi$. Using that $A$ is commutative, this further implies $a^*c^*\xi=b^* c^*\xi$ for every $c\in A$, and in turn we must have $a^*=b^*$ or $a=b$, a contradiction. 

Some additional partial results on the existence of strictly cyclic vector can be found in \cite{shields1974}.

\section{Null projections and the  Riesz projection}\label{S:riesznull}
In this section, we consider certain projections arising from convex sets of states on a $\rC^*$-algebra, and aim to understand when they are null,
 or more generally when projections annihilating the `absolutely continuous’ functionals are null.  Let us be more precise.

Let $B$ be a 
 $\rC^*$-algebra. Let $\phi$ and $\psi$ be continuous linear functionals on $B$. We say that $\phi$ is \emph{absolutely continuous} with respect to $\psi$ if $|\phi|(x^*x)=0$ for every $x\in B^{**}$ satisfying $|\psi|(x^*x)=0$. Equivalently, this means that the support projection of $|\phi|$ is dominated by the support projection of $|\psi|$. Here, by the support projection of a state $\theta$ on $B$, we mean the unique projection $s_\theta\in B^{**}$ such that
\[
\{x\in B^{**}:\theta(x^*x)=0\}=B^{**}(1-s_\theta).
\]
When the support projections of $|\phi|$ and $|\psi|$ are instead orthogonal, we say that $\phi$ and $\psi$ are \emph{singular}. The reader may consult \cite[Section 2]{CT2023Henkin} for more detail on absolute continuity and singularity  for states.

Let now $\Delta$ be a norm-closed convex set of states on $B$.  We let $\AC(\Delta)\subset B^*$ denote the set of those functionals that are absolutely continuous with respect to some element of $\Delta$. Similarly, we let $\SG(\Delta)$ denote the set of those functionals that are mutually singular with respect to every element of $\Delta$.   By \cite[Theorem 3.5]{CT2023Henkin}, there is a projection $r_\Delta\in B^{**}$ with the property that $\AC(\Delta)=((1-r_\Delta)B^{**})_\perp$ and $\SG(\Delta)=(r_\Delta B^{**})_\perp$. We sometimes refer to $r_\Delta$ as the \emph{Riesz projection} corresponding to $\Delta$. The following easy observation will be used several times below.

\begin{lemma}\label{L:SG}
Let $\phi,\psi\in B^*$ such that $\phi$ is absolutely continuous with respect to $\psi$. Then, the following statements hold.
\begin{enumerate}[{\rm (i)}]
\item If $\psi\in \AC(\Delta)$, then $\phi\in \AC(\Delta)$.
\item If $\psi\in \SG(\Delta)$, then $\phi\in \SG(\Delta)$.
\end{enumerate}
\end{lemma}
\begin{proof}
(i) is trivial. To see (ii), assume that $\psi\in \SG(\Delta)$. Fix $b\in B$ and put $\rho=\psi(\cdot b)$. Then, $\rho(r_\Delta x)=\psi(r_\Delta x b)=0$ for each $x\in B^{**}$, since $\psi\in (r_\Delta B^{**})_\perp$. This shows that $\rho\in (r_\Delta B^{**})_\perp=\SG(\Delta)$. Further, the equality  $(r_\Delta B^{**})_\perp=\SG(\Delta)$ shows that $\SG(\Delta)$ is a norm-closed subspace, so from the previous argument we conclude that $\phi\in \SG(\Delta)$ upon invoking \cite[Theorem 2.7]{CT2023Henkin}.
\end{proof}

Let $A\subset B$ be a subalgebra. Henceforth we assume that this inclusion is unital for simplicity. 
A unifying motif behind this section
is to determine when the Riesz projection $r_\Delta$ is $A$-null. In doing so, we will highlight connections with many different properties of interest from the point of view of noncommutative function theory.

First, we tackle the issue of determining when $r_\Delta$ lies in $A^{\perp\perp}$. For this purpose, we need the following notion. We say that $A$ has the \emph{F.\&M. Riesz property in $B$ with respect to $\Delta$} if, whenever $\phi\in A^\perp$, we have that $\phi(r_\Delta \cdot)\in A^{\perp}$. In this case, we also have $\phi((1-r_\Delta) \cdot)\in A^{\perp}$. In other words, this property says that the absolutely continuous and singular parts (with respect to $\Delta$) of a functional on $B$ annihilating $A$ must also annihilate $A$. This can be reformulated as follows.

 \begin{lemma}\label{L:FMRiesz}
The following statements are equivalent.
 \begin{enumerate}[{\rm (i)}]
\item $A$ has the F.\&M. Riesz property in $B$ with respect to $\Delta$.
\item The Riesz projection $r_\Delta$ lies in $A^{\perp\perp}$.
\end{enumerate} 
 \end{lemma}
 \begin{proof}
 (i) $\Rightarrow$ (ii): Let $\phi\in A^\perp$. Then, $\phi(r_\Delta\cdot)\in A^\perp$ by
  the F.\&M. Riesz property. Evaluating at $1\in A$, we obtain in particular $\phi(r_\Delta)=0$. We conclude that $r_\Delta\in A^{\perp\perp}$.
 
 (ii) $\Rightarrow$ (i): Let $\phi\in A^\perp$. Then, $\phi$ annihilates $A^{\perp\perp}$, and in particular it annihilates $r_\Delta A$ since $r_\Delta\in A^{\perp\perp}
 = \bar{A}^{w*}$. This implies in turn that $\phi(r_\Delta \cdot)$ annihilates $A$.
 \end{proof}

Notice that if 
\begin{equation}\label{Eq:SFM}
A^\perp\subset \AC(\Delta)
\end{equation} 
then trivially $A$ has the F.\&M. Riesz property in $B$ with respect to $\Delta$. We therefore refer to
condition \eqref{Eq:SFM} as the \emph{strong F.\&M. Riesz property in $B$ with respect to $\Delta$} for $A$. 
For example, the original  F.\&M. Riesz theorem shows that the disk algebra
$A(\bD)$ has the  strong F.\&M. Riesz property in $\rC(\bT)$  with respect to arclength measure. 
This can also be encoded in terms of the Riesz projection, as we show next.

 \begin{lemma}\label{L:SFMRiesz}
The following statements are equivalent.
 \begin{enumerate}[{\rm (i)}]
\item $A$ has the  strong F.\&M. Riesz property in $B$ with respect to $\Delta$.
\item The Riesz projection $1-r_\Delta$ is $A$-null.
\end{enumerate} 
 \end{lemma}
 \begin{proof}
(i) $\Rightarrow$ (ii): The assumption immediately implies that
$
\{|\phi|:\phi\in A^\perp\}\subset \AC(\Delta).
$
Given $\phi\in A^\perp$, we thus find $|\phi|=|\phi|(r_\Delta\cdot)$ and $|\phi|(1-r_\Delta)=0$.

(ii) $\Rightarrow$ (i): 
Let $\phi\in A^\perp$. By assumption, $|\phi|(1-r_\Delta)=0$, so that $\phi((1-r_\Delta)\cdot)=0$ by Lemma \ref{L:stateineq}. This shows that $\phi\in \AC(\Delta)$.
\end{proof}

Our next result will provide several conditions that interpolate between Lemmas \ref{L:FMRiesz} and \ref{L:SFMRiesz}. For this purpose, we need a technical result from \cite{CT2023Henkin}. A projection $q\in B^{**}$ will be said to be of \emph{type $F$} if it is the supremum of an arbitrary collection of closed projections in $B^{**}$. If the collection can be chosen to be countable, we say that $q$ is of \emph{type $F_\sigma$}. If the collection can be chosen to be finite, we say that $q$ is of \emph{type $F_0$}. If $B$ is not commutative, such a projection need not be closed \cite[Example II.6]{akemann1969}

\begin{lemma}\label{L:Rainwater}
Assume that $\Delta$ is closed in the weak-$*$ topology of $B^*$. Let $\psi$ be a state in $\SG(\Delta)$. Then, there is a type $F_\sigma$ projection $q\in B^{**}$ annihilating $\AC(\Delta)$ and satisfying $\psi(q)=1$.  
\end{lemma}
\begin{proof}
By \cite[Theorem 3.7]{CT2023Henkin}, there a type $F_\sigma$ projection $q\in B^{**}$ that  annihilates $\Delta$  and satisfies $\psi(q)=1$. Furthermore, an application of \cite[Lemma 2.4]{CT2023Henkin}, shows that in fact $q$ annihilates $\AC(\Delta)$. \end{proof}

We say that $A$ has the \emph{Forelli property in $B$ with respect to $\Delta$} if 
for every type $F$ projection $q\in B^{**}$ lying in $\AC(\Delta)^{\perp}$,  there is a contractive net $(a_i)$ in $A$ such that $(q a_i )$ converges to $0$ in the weak-$*$ topology of $B^{**}$, and $(\phi(1-a_i))$ converges to $0$ for every $\phi\in \AC(\Delta)$.  If this property only holds for type $F_\sigma$ projections, then we say instead that $A$ has the \emph{$\sigma$-Forelli property in $B$ with respect to $\Delta$}. If this property only holds for type $F_0$ projections, then we say instead that $A$ has the \emph{finite Forelli property in $B$ with respect to $\Delta$}. 

These properties are inspired by a result of Forelli (see for instance \cite[Lemma 9.5.5]{rudin2008}), which explains our choice of terminology. We now arrive at another main result of the paper, which connects nullity of the Riesz projection with the various properties introduced above.

\begin{theorem}\label{T:ForelliFM}
Let $B$ be a unital $\rC^*$-algebra and let $A\subset B$ be a unital norm-closed subalgebra. Let $\Delta$ be a norm-closed convex subset of the state space of $B$. Consider the following statements.
\begin{enumerate}[{\rm (i)}]
\item The projection $1-r_\Delta$ is $A$-null.
\item $A$ has the strong F.\&M. Riesz property in $B$ with respect to $\Delta$.
\item All projections in $\AC(\Delta)^\perp$ are $A$-null.
\item All closed projections in $\AC(\Delta)^\perp$ are $A$-null.
\item Let $p,q\in B^{**}$ be closed projections. If $p\in \AC(\Delta)^\perp$, then $pq\in A^{\perp\perp}$. 
\item All closed projections in $\AC(\Delta)^\perp$ are dominated by a projection in $\AC(\Delta)^\perp\cap A^{\perp\perp}$.
\item $A$ has the Forelli property in $B$ with respect to $\Delta$.
\item $A$ has the $\sigma$-Forelli property in $B$ with respect to $\Delta$.
\item $A$ has the finite Forelli property in $B$ with respect to $\Delta$.
\item $A$ has the F.\&M. Riesz property in $B$ with respect to $\Delta$.
\item The projection $1-r_\Delta$ lies in $A^{\perp\perp}$.
\end{enumerate}
Then,
\[
\begin{array}{ccccccccccccccc}
{\rm (i)} & \Leftrightarrow & {\rm (ii)}  &\Leftrightarrow & {\rm (iii)} & \Rightarrow & {\rm (iv)} &\Leftrightarrow & {\rm (v)} &&&&&&\\
&&&&&&&&\Downarrow &&&&&& \\
&&&&{\rm (xi)}&\Leftrightarrow &{\rm (x)}& \Rightarrow & {\rm (vi)} & \Rightarrow & {\rm (vii)} \Rightarrow & {\rm (viii)} \Rightarrow & {\rm (ix)} 
\end{array}
\]
If we assume in addition that $\Delta$ is weak-$*$ closed in $B^*$, then 
\[
\begin{array}{ccccccccccc}
&&{\rm (i)} & \Leftrightarrow & {\rm (ii)}  &\Leftrightarrow & {\rm (iii)} & \Leftrightarrow & {\rm (iv)} &\Leftrightarrow & {\rm (v)} \\
&&&&&&&&&&\Downarrow \\
{\rm (xi)}&\Leftrightarrow &{\rm (x)}&\Leftrightarrow&{\rm (ix)}&\Leftrightarrow &{\rm (viii)}&\Leftrightarrow&{\rm (vii)} & \Leftrightarrow & {\rm (vi)}
\end{array}
\]
\end{theorem}
\begin{proof}
(iii) $\Leftrightarrow$ (i): This is trivial since $1-r_\Delta\in \AC(\Delta)^\perp$.

(i) $\Leftrightarrow$ (ii): This is Lemma \ref{L:SFMRiesz}.

(i) $\Rightarrow$ (iii): Any projection lying in $\AC(\Delta)^\perp$ is dominated by $1-r_\Delta$, and hence  is $A$-null.

(iii) $\Rightarrow$ (iv): This is trivial.

(iv) $\Rightarrow$ (v): Let $p,q\in B^{**}$ be closed projections such that $p\in \AC(\Delta)^\perp$. By assumption, $p$ is $A$-null, so that $\phi(pq)=0$ for every $\phi\in A^\perp$ by Lemma \ref{L:stateineq}. This means that $pq\in A^{\perp\perp}$.

(v) $\Rightarrow$ (iv): Apply Proposition \ref{P:nullproj}.

(v) $\Rightarrow$ (vi): Take $q=1$.

(vi) $\Rightarrow$ (vii): Let $q\in B^{**}$ be a type $F$ projection in $\AC(\Delta)^{\perp}$. Thus, there is a collection $\F\subset  B^{**}$ of closed projections such that $q=\bigvee_{p\in \F}p$. Since $q\in\AC(\Delta)^{\perp}$, we see that  $q\leq 1-r_\Delta$.  Hence $p\leq 1-r_\Delta$ and  so $p\in \AC(\Delta)^\perp$ for each $p\in \F$. By assumption, for each $p\in \F$ there is a projection $r_p\in \AC(\Delta)^\perp\cap A^{\perp\perp}$ such that $p\leq r_p$. Put $q'=\bigvee_{p\in \F}r_p$, which still lies in
 $A^{\perp\perp}$.  Furthermore, since $r_p\in \AC(\Delta)^\perp$, we find $r_p\leq 1-r_\Delta$ for each $p\in F$. Thus, $q'\leq 1-r_\Delta$ and $q'\in \AC(\Delta)^\perp$. Now, by virtue of Goldstine's lemma, because $1-q'\in A^{\perp\perp}$,  we may find a contractive net $(a_i)$ in $A$ converging to $1-q'$ in the weak-$*$ topology of $B^{**}$. In particular, because $q\leq q'$, this means that $(qa_i)$ converges to $0$ in the weak-$*$ topology of $B^{**}$. Further, if $\phi\in \AC(\Delta)$, then since $q'\in \AC(\Delta)^\perp$ we have $\phi(q')=0$ and
\[
\lim_i \phi(1-a_i)=\phi(q')=0.
\]

(vii) $\Rightarrow$ (viii) $\Rightarrow$ (ix): This is trivial.

(x) $\Leftrightarrow$ (xi): This is a direct consequence of Lemma \ref{L:FMRiesz} since $A$ is unital.

(xi) $\Rightarrow$ (vi): If $p\in \AC(\Delta)^\perp$ is any projection, then $p\leq 1-r_\Delta$, and $1-r_\Delta\in \AC(\Delta)^\perp\cap A^{\perp\perp}$ by assumption.

For the rest of the proof, we assume that $\Delta$ is closed in the weak-$*$ topology of $B^*$.

(ix) $\Rightarrow$ (x): Fix a bounded linear functional $\phi:B\to \bC$ annihilating $A$. Put $\psi=\phi((1-r_\Delta)\cdot)$ and $\theta=\phi(r_\Delta \cdot)$, so that $\phi=\theta+\psi$ with $\theta\in \AC(\Delta)$ and $\psi\in \SG(\Delta)$. The goal is to show that $\theta$ and $\psi$ both annihilate $A$ as well, for which it clearly suffices to show that $\theta$ annihilates $A$.  Fix $a_0\in A$. 
In showing that $\theta(a_0)=0$, we may assume that $|\psi|$ is a state, upon scaling if necessary. Note that $|\psi|\in \SG(\Delta)$ by Lemma \ref{L:SG}.
By Lemma \ref{L:Rainwater}, there is a type $F_\sigma$ projection
$r\in B^{**}$ annihilating $\AC(\Delta)$ such that 
$
|\psi|(r) = 1.
$ 
Given $\eps>0$, we may thus find a type $F_0$ projection $q\leq r$ such that $|\psi|(1-q)<\eps$. In particular, $q$ still annihilates $\AC(\Delta)$. Applying  Lemma \ref{L:stateineq}, we obtain
\begin{equation}\label{Eq:Forelli1}
|\psi((1-q)x)|< \|x\| \sqrt{\eps}, \quad x\in B^{**}.
\end{equation}
On the other hand, using the fact that $\theta\in \AC(\Delta)$ along with \cite[Lemma 2.5]{CT2023Henkin} we see that $\theta(\cdot a_0)\in \AC(\Delta)$ as well. 
Using the  finite Forelli property, we find  $a \in A$ with $\|a\|\leq 1$ such that 
\begin{equation}\label{Eq:Forelli2}
|\theta((1-a)a_0)|<\eps
\end{equation}
and $|\psi(qaa_0)|<\eps$. In light of \eqref{Eq:Forelli1}, this last inequality implies
\[
|\psi(aa_0)| <\|a_0\|\sqrt{\eps}+\eps.
\]
In turn, we have $\psi=\phi-\theta$ and  $\phi\in A^\perp$, so that
\begin{equation}\label{Eq:Forelli3}
|\theta(aa_0)|=|\psi(aa_0)|<\|a_0\|\sqrt{\eps}+\eps.
\end{equation}
Combining \eqref{Eq:Forelli2} and \eqref{Eq:Forelli3}, we find $|\theta(a_0)|<\|a_0\|\sqrt{\eps}+2\eps$.
 Letting $\eps\to 0$, we see that $\theta(a_0)=0$ as desired.

(iv) $\Rightarrow$ (ii): Let $\phi\in A^\perp$ and put $\psi=\phi((1-r_\Delta)\cdot)$. We must show that $\psi=0$. Assume towards a contradiction that $\psi\neq 0$. Then, $|\psi|$ lies in $\SG(\Delta)$, and upon scaling if necessary, we may assume $|\psi|$ is a state.  By Lemma \ref{L:Rainwater}, there is a collection $\F\subset B^{**}$ of closed projections such that the projection $q=\bigvee_{p\in \F} p$ annihilates $\AC(\Delta)$ and satisfies $|\psi|(q)=1$. As before in the proof of (vii), 
 it is readily seen that $p\in \AC(\Delta)^{\perp}$ for every $p\in \F$. The assumption then shows that $p$ is $A$-null for each $p\in \F$, and 
 thus $q$ is $A$-null. On the other hand, we already know that (iv)$\Rightarrow$(x)$\Leftrightarrow$(xi), so $\psi\in A^\perp$ and $|\psi|(q)=0$, which is absurd.
\end{proof}

The proof of (ix) $\Rightarrow$ (x) above is inspired by that of \cite[Theorem 9.5.6]{rudin2008}.   

\medskip

{\bf Remark.} The reader may notice that there is a notable omission in the long list in the statement of the last theorem, namely that all closed projections in AC$(\Delta)^\perp$ are in $A^{\perp \perp}$.  This seems to be a very interesting condition.    It is equivalent to (iv) in the commutative case, but we are not sure about the general case.  
Indeed if $B = \rC(X)$ for some compact Hausdorff space $X$,
then 
this condition is simply saying that a closed set  $E$ in $X$ is a generalized $A$-peak set if each of its closed subsets has $\mu$-measure zero  for every probability measure $\mu$ in $\Delta$.  
That is, a closed set  $E\subset X$ is a generalized $A$-peak set if  
$| \mu | (E) = 0$ 
for every probability measure $\mu$ in $\Delta$.  This clearly implies that 
every closed subset of such a set $E$ is a generalized peak set, which implies by Proposition \ref{P:nullproj} or the discussion above it that $\chi_E$ is $A$-null. 
 That is,
$E$ is a peak-interpolation set.

When $\Delta = \{ \mu \}$ for some regular Borel probability measure $\mu$ on $X$,  and every closed set of $\mu$-measure zero is a peak set for $A$,  it follows that every closed set of $\mu$-measure zero is $A$-null.  This was showed to us by Alexander Izzo, and it inspired several of the results in Section 2.

\section{Existence of Lebesgue decompositions}\label{S:Leb}

In this section, we apply some of the ideas from Theorem \ref{T:ForelliFM} to prove the existence of certain decompositions for the dual space of an operator algebra.

\subsection{Von Neumann algebraic preliminaries}
We begin with some general facts about von Neumann algebras that will be needed throughout the remainder of the paper.

Given a von Neumann algebra $M$, we let $M_*\subset M^*$ denote the norm-closed subspace of weak-$*$ continuous linear functionals on $M$.  The following is actually a general principle for dual Banach spaces
that is well known in certain quarters.   It is a necessary and sufficient condition for a kind of `Kaplansky density'. 

\begin{lemma}\label{L:Henkinclosed}
Let $M$ be a von Neumann algebra. Let $A\subset M$ be a subspace and let $W\subset M$ denote its weak-$*$ closure. Let $\rho:M_*\to A^*$ denote the restriction map. Then, $\rho$ has closed range if and only if there is a constant $r>0$ such that the closed unit ball of $W$ is contained in the weak-$*$ closure of $\{a\in A:\|a\|\leq r\}$.
\end{lemma}
\begin{proof}
Assume first that $\rho$ has closed range.  By the open mapping theorem,  there is $r>0$ such that for every $\omega\in M_*$, there is $\omega'\in M_*$ with $\rho(\omega)=\rho(\omega')$ and  $\|\omega'\|\leq r\|\rho(\omega)\|$. Then, for each $\omega\in M_*$ we see that
\begin{align*}
\sup\{|\omega(a)|:a\in A,\|a\|\leq r\}&=r \|\rho(\omega)\|\geq \|\omega'\|\geq \|\omega'|_W\|=\|\omega|_W\|\\
&=\sup\{|\omega(a)|:a\in W, \|a\|\leq 1\}.
\end{align*}
By the convex separation theorem, we conclude that the weak-$*$ closure of $\{a\in A:\|a\|\leq r\}$ contains the closed unit ball of $W$.

Conversely, assume that there is $r>0$ such that the closed unit ball of $W$ is contained in the weak-$*$ closure of $\{a\in A:\|a\|\leq r\}$. Let $\omega\in M_*$. It follows that
\[
r\|\rho(\omega)\|=\sup\{|\omega(a)|:a\in A,\|a\|\leq r\}\geq \|\omega|_W\|.
\]
On the other hand, the predual of $W$ is isometrically isomorphic to $M_*/W_\perp$, so that there is $\omega'\in M_*$ such that $\omega'|_W=\omega|_W$ and $\|\omega'\|\leq 2 \|\omega|_W\|$. In particular, we find $\rho(\omega')=\rho(\omega)$ and $\|\omega'\|\leq 2r\|\rho(\omega)\|$, so the range of $\omega$ is closed.
\end{proof}

There is a central projection $z\in M^{**}$ such that $M_*=M^* z$ \cite[Theorem III.2.14]{takesaki2002}. An important fact is that there is a normal $*$-isomorphism $\Omega: M\to M^{**}z$ such that $\Omega(a)=az$ for every $a\in M$. Next, let $\phi$ be a normal state on $M$,
and let $\widehat\phi:M^{**}\to\bC$ denote its unique normal extension. Consider the weak-$*$ closed left ideals
\[
L_\phi=\{a\in M:\phi(a^*a)=0\} \qand L_{\widehat\phi}=\{x\in M^{**}:\widehat\phi(x^*x)=0\}.
\]
Then, it is readily verified that $\phi=\widehat\phi\circ \Omega$ and that
$
L_{\widehat\phi}=\Omega(L_\phi)\oplus M^{**}(1-z).
$
In particular, this gives a simple relationship between what we call the support projection $s_\phi$  of $\phi$, and another notion of support found in the von Neumann algebras literature. 
Indeed, our convention in this paper is that
$
L_{\widehat\phi}=M^{**}(1-s_\phi).
$
On the other hand, there is a projection $e_\phi\in M$ such that $L_\phi=M(1-e_\phi)$. It follows that $s_\phi=ze_\phi$. Furthermore, given another normal state $\psi$ on $M$, we see that $L_\phi\subset L_\psi$ if and only if $L_{\widehat\phi}\subset L_{\widehat\psi}$, so that our notion of absolute continuity  coincides with the corresponding notion from von Neumann algebra theory for such states.

We say that a normal state $\phi$ is \emph{faithful} if $L_\phi=\{0\}$. We now give an alternative description of $z$.

\begin{lemma}\label{L:predualphi}
Let $M$ be a von Neumann algebra with a faithful normal state $\phi$. 
Then, $s_\phi=z$ and $M_*=\AC(\varphi)$.
\end{lemma}
\begin{proof}  
Since $\phi$ is normal, we have $\phi(z)=1$ so $z\geq s_\phi$. Next, by the discussion above we have $L_{\widehat\phi}z=\Omega(L_\phi)=\{0\}$ since $\phi$ is faithful. In particular, $(1-s_\phi)z=0$, so that $z\leq s_\phi$.  Therefore, $z=s_\phi$.
It follows immediately from this that $M_* \subset  {\rm AC}(\phi)$.  Conversely,  if $\psi \in {\rm AC}(\phi)$ then $|\psi| (1-z) = 0$ since $\phi(1-z) = 0$, so indeed $\psi\in M_*=M^*z$ by Lemma \ref{L:stateineq}.
\end{proof}

More generally, we can describe $M_*$ in terms of absolute continuity, as follows.
We  denote by  $S_n(M)$ the set of normal states on $M$.  Assume now that $M$ contains a weak-$*$ dense unital $\rC^*$-subalgebra $B$. Let $\rho:M^*\to B^*$ denote the restriction map. By the Kaplansky density theorem we see that $\rho$ is isometric on $M_*$. In particular, this means that the sets $(M_*)|_B=\rho(M_*)$ and  $\rho( S_n(M))$ are norm-closed. Put $\Delta=\rho(S_n(M))$; this is the norm-closed convex set
 in $B^*$
 consisting of those elements of the form $\omega|_B$, where $\omega$ is a normal state on $M$. We now record a simple fact.

\begin{lemma}\label{L:restnormal}
We have that $\AC(S_n(M)))=M_*$ and $\AC(\Delta)=(M_*)|_B$.
\end{lemma}
\begin{proof}
Given $\omega\in M_*$ and $a\in M$, we have that $\omega(\cdot a)\in M_*$. Because both $M_*$ and $(M_*)|_B$ are norm-closed, we then find $\AC(S_n(M))\subset M_*$ and $\AC(\Delta)\subset (M_*)|_B$ by virtue of \cite[Theorem 2.7]{CT2023Henkin}.  Conversely, let $\phi\in M_*$ with $\|\phi\|=1$. We have $\phi(z\cdot)=\phi$. Now, there is a partial isometry $v\in M$ such that $|\phi|=\phi(\cdot v)$. It follows that $|\phi|(z\cdot )=|\phi|$, whence $|\phi|\in S_n(M)$ and $\phi\in \AC(S_n(M))$. We conclude that  $M_*\subset \AC(S_n(M))$, which in turn plainly implies $(M_*)|_B\subset \AC(\Delta)$.
\end{proof}

\subsection{Lebesgue decompositions}

Let $B$ be a unital $\rC^*$-algebra and let $A\subset B$ be a unital subalgebra. Given a subset $X\subset B^*$, we let $X|_A$ denote the set of restrictions to $A$ of functionals in $X$. Following \cite{CT2023Henkin}, we say that $X$ is a \emph{left band} if, whenever a functional $\phi\in B^*$ is absolutely continuous with respect to some element of $X$, then it must be that $\phi$ itself lies in $X$.

Next, let $M$ be a von Neumann algebra containing $B$ as a weak-$*$ dense unital $\rC^*$-subalgebra. 
We say that $A$ admits a \emph{Lebesgue decomposition} relative to $M$ if $(M_*)|_A$ is norm-closed and there is a  (necessarily unique) subspace $\Sigma\subset A^*$ such that $A^*=(M_*)|_A\oplus_1 \Sigma$.  

It follows from \cite[Lemma 3.4]{CH2025} that $B$ itself always admits a Lebesgue decomposition relative to $M$. By contrast, 
there is a sense in which the existence of such a decomposition for $A$ is somewhat rare (see \cite{CH2025} for more detail). Combining Lemma \ref{L:restnormal} with \cite[Theorem 2.4]{CH2025}, we see that  if $A$ has the F.\& M. Riesz property in $B$ with respect to the convex set of restrictions to $B$ of the normal states on $M$, then $A$ admits a Lebesgue decomposition relative to $M$ that is `compatible' (in a precise sense) with that of $B$. We can now give a different sufficient condition, where we don't insist on this compatibility condition.

\begin{corollary}\label{C:Lebdecomp}
Let $M$ be a von Neumann algebra. Let $B\subset M$ be a weak-$*$ dense unital $\rC^*$-subalgebra, and let $A\subset B$ be a unital commutative subalgebra. Let
\[
X=\{\psi\in B^*:\psi|_A\in (M_*)|_A\}.
\]
If  $X$ is a
norm-closed left band, then $A$ admits a Lebesgue decomposition relative to $M$.
\end{corollary}
\begin{proof}
Let $\Gamma\subset B^*$ denote the norm-closure of the convex hull of $\{|\psi|:\psi\in X, \|\psi\|=1\}$.  Let $r_\Gamma\in B^{**}$ denote the corresponding Riesz projection \cite[Theorem 3.5]{CT2023Henkin}. By assumption, $X$ is norm-closed and a left band, so we see that $X=\AC(\Gamma)$. On the other hand, it is clear that $A^\perp\subset X$, so in fact $A^\perp \subset \AC(\Gamma)$. We conclude that $A$ has the strong F.\& M. Riesz property in $B$ with respect to $\Gamma$. Hence, Theorem \ref{T:ForelliFM} implies that  $r_\Gamma\in A^{\perp\perp}$. In particular, $r_\Gamma$ commutes with $A$ since $A$ is assumed to be commutative. It thus follows that $A^*=(A^{**} r_\Gamma)_\perp \oplus_1 (A^{**}(1-r_\Gamma))_\perp$. Hence, it suffices to check that $(A^{**}(1-r_\Gamma))_\perp=(M_*)|_A$.

Let $\phi\in (M_*)|_A$ and choose $\omega\in M_*$ such that $\omega|_A=\phi$. Plainly, 
$\omega|
_{B} \in X= \AC(\Gamma)$, so by construction of the Riesz projection, we have $\omega((1-r_\Gamma)\cdot)=0$. Using that $r_\Gamma\in A^{\perp\perp}$, we find
\[
\phi((1-r_\Gamma)a)=\omega((1-r_\Gamma)a)=0, \quad a\in A
\]
so that $\phi\in (A^{**}(1-r_\Gamma))_\perp$. Conversely let $\phi\in A^*$ such that $\phi((1-r_\Gamma)\cdot)=0$. By the Hahn--Banach theorem, there is  $\psi\in B^*$  such that $\psi|_A=\phi$. Since $r_\Gamma\in A^{\perp\perp}$, this implies that $\psi( r_\Gamma \cdot)|_A=\phi(r_\Gamma \cdot)=\phi$. Put $\psi'=\psi(r_\Gamma\cdot )\in B^*$, which then lies  in $\AC(\Gamma)=X$. Hence, there is $\omega\in M_*$ such that $\omega|_A=\psi'|_A=\phi$, so that $\phi\in (M_*)|_A$ as desired.
\end{proof}

We now illustrate how this result applies in some classical settings.

\begin{example}\label{E:ballalg}
Let $d\geq 1$ be an integer, let $\bB_d\subset\bC_d$ be the open unit ball, and let $\bS_d$ denote its topological boundary, the unit sphere. Let $\sigma$ denote surface measure on $\bS_d$. Let $M=L^\infty(\bS_d,\sigma), B=\rC(\bS_d)$ and let $A\subset B$ denote the ball algebra, i.e.\ the norm-closure of the polynomials. 
It is well known that the weak-$*$ closure of $A$ in $M$ is simply $H^\infty(\bB_d)$,  by making the usual identification via radial boundary values. Further, given $f\in H^\infty(\bB_d)$ and $0<r<1$, we define $f_r:\ol{\bB_d} \to \bC$ as $f_r(z)=f(rz)$. Then, $f_r\in A$, $\|f_r\|\leq \|f\|$ and the net $(f_r)$ converges to $f$ in the weak-$*$ topology of $M$. We infer that the set
\[
X=\{\psi\in B^*:\psi|_{A}\in (M_*)|_{A}\}
\]
is norm-closed by Lemma \ref{L:Henkinclosed}. In turn, $X$
is a left band by Henkin's theorem \cite[Theorem 9.3.1]{rudin2008}, 
Thus, $A$ admits a Lebesgue decomposition relative to $M$ by Corollary \ref{C:Lebdecomp}; this also follows from \cite[Section 9.8]{rudin2008}. 

Next, we note that  $B$ admits a Lebesgue decomposition relative to $M$ by  \cite[Lemma 3.4]{CH2025}. Nevertheless, the Lebesgue decompositions of $A$ and $B$ relative to $M$ are not compatible, by \cite[Theorem 4.8 and Proposition 6.3]{CH2025}.  \qed
\end{example}

\begin{example}\label{E:AH}
Let $\H$ be a Hilbert function space of the type considered  in Subsection \ref{SS:HFS}. Let $\A(\H)\subset \T(\H)\subset B(\H)$ be the operator algebras defined therein. Consider the set
\[
X=\{\psi\in \T(\H)^*:\psi|_{\A(\H)}\in (B(\H)_*)|_{\A(\H)}\}.
\]
This is the set of $\M(\H)$-functionals examined in Section \ref{S:A(H)}. As explained there, this coincides with the set $\mathscr{B}$  from  \cite[Theorem 5.7]{CT2023Henkin}, so we see that $X$ is a norm-closed left band.  Thus, $\A(\H)$ admits a Lebesgue decomposition relative to $B(\H)$ by Corollary \ref{C:Lebdecomp}. 

Next, recall that  $\T(\H)$ contains the ideal of compact operators on $\H$, whence it admits a Lebesgue decomposition relative to $B(\H)$ by  \cite[Lemma 3.4]{CH2025}. Nevertheless, the Lebesgue decompositions of 
$\A(\H)$ and $\T(\H)$ relative to $B(\H)$ are not compatible, by \cite[Theorem 4.8 and Proposition 6.1]{CH2025}.
\qed
\end{example}

We emphasize again that the incompatibility of the Lebesgue decompositions in the previous two examples shows in particular that Corollary \ref{C:Lebdecomp} is applicable in some cases where \cite[Theorem 2.4]{CH2025} is not.

\section{The F.\& M. Riesz properties}\label{S:FM}
It is a natural question whether the statements in Theorem \ref{T:ForelliFM} are ever all equivalent, perhaps under some additional assumptions. One way of answering this question is to determine when the F.\& M. Riesz property is equivalent to its stronger variant. We explore this issue in this section.

Let $B$ be a unital $\rC^*$-algebra and let $A\subset B$ be a unital subalgebra. 
Let $\Delta\subset B^*$ be a norm-closed, convex set of states on $B$.  Define
\[
E(\Delta)=\{\psi\in B^*:\psi|_A\in  \AC(\Delta)|_A\}.
\]
This set will play an analogous role to the set of functionals whose restrictions to $A$ are ``Henkin" (see e.g.\ Section 8.1 below).
We now show that the F.\& M. Riesz property implies a certain rigidity property for left bands containing $\Delta$.  
 Its relation to the 
 classical Cole-Range theorem \cite[Theorem 9.6.1]{rudin2008} will be thoroughly 
explained at the end of this section.

\begin{theorem}\label{T:colerange}
Assume that $A$ has the F.\& M. Riesz property in $B$ with respect to $\Delta$. If $X\subset B^*$ is a left band satisfying $\Delta\subset X\subset E(\Delta)$, then $X=\AC(\Delta)$.
\end{theorem} 
 \begin{proof}
 Since $X$ is assumed to be a left band containing $\Delta$, we necessarily have $\AC(\Delta)\subset X$. Conversely, fix $\phi\in X$. We see that $\phi( r_\Delta \cdot)\in \AC(\Delta)\subset X$, so it suffices to show that the functional $\psi=\phi( (1-r_\Delta)\cdot)$ is zero. To see this, fix $b\in B$. Since $X$  is a subspace, $\psi=\phi-\phi( r_\Delta \cdot)\in X$. Further, since $X$ is a left band again, we see that $\psi( \cdot b)\in X$ by
 \cite[Lemma  2.7]{CT2023Henkin}. By assumption, we have $X\subset E(\Delta)$.  Using the definition of $E(\Delta)$ we may  find 
 $\nu\in A^\perp$ and $\alpha\in \AC(\Delta)$ such that $\psi( \cdot b )=\alpha+\nu$. On the other hand, $\psi(\cdot b)\in \SG(\Delta)$ by Lemma \ref{L:SG}. Since we have $\nu=-\alpha+\psi(\cdot b)$, the F.\& M. Riesz property implies 
 that $\psi(\cdot b)\in A^{\perp}$, so in particular $\psi(b)=0$. Consequently, $\psi=0$ as desired.
 \end{proof}
 
 We can now reformulate the strong F.\& M. Riesz property.   

\begin{corollary} \label{C:FMs}  Let $B$ be a unital $\rC^*$-algebra and let $A\subset B$ be a unital  subalgebra. Let $\Delta$ be a norm-closed convex subset of the state space of $B$.  Assume that $A$ has the F.\& M. Riesz property in $B$ with respect to $\Delta$. 
Then the  following statements are equivalent.  \begin{enumerate}[{\rm (i)}]
\item  $A$ has the strong F.\& M. Riesz property in $B$ with respect to $\Delta$.
\item  $E(\Delta) =  \AC(\Delta)$. 
\item $E(\Delta)$ is a left band.
\end{enumerate}
 \end{corollary}
\begin{proof} 

 (i) $\Rightarrow$ (ii): The inclusion $\AC(\Delta)\subset E(\Delta)$ is trivial. Conversely, assume that $\psi\in E(\Delta)$. By definition, this means that there is $\phi\in \AC(\Delta)$ such that $\psi|_A=\phi|_A$. Now, $\psi-\phi\in A^\perp$ so by assumption $\psi-\phi\in \AC(\Delta)$, which forces $\psi\in \AC(\Delta)$.
 
 (ii) $\Rightarrow$ (iii): This is trivial.
 
 (iii) $\Rightarrow$ (i):  This follows immediately from Theorem \ref{T:colerange}.
\end{proof}

We remark that by using the F.\& M. Riesz property,  it is easy to see that the  strong F.\& M. Riesz property with respect to $\Delta$ is also  equivalent to $\SG(\Delta)\cap A^\perp=\{0\}$.

  We illustrate this result by applying it to a classical example.

\begin{example}\label{E:Hinfband}
Let $\lambda$ denote arclength measure on the unit circle $\bT\subset \bC$. By means of radial boundary values, we may embed the algebra $H^\infty(\bD)$ of bounded holomorphic functions on the open unit disc as a unital norm-closed subalgebra of $L^\infty(\bT,\lambda)$. Choosing $\Delta=\{\lambda\}$, 
Lemma \ref{L:predualphi} implies that $\AC(\Delta)$ coincides with the weak-$*$ continuous linear functionals on $L^\infty(\bT,\lambda)$. 
Hence, $E(\Delta)$ coincides with the set $X$ appearing in Corollary \ref{C:Lebdecomp}.

Consider now the Douglas algebra $H^\infty(\bD)+\rC(\bT)\subset L^\infty(\bT,\lambda)$, which is known to be proper and norm-closed (see for instance \cite[Chapter 6]{douglas1998}). Hence, there is a non-zero bounded linear functional $\phi$ on $L^\infty(\bT,\lambda)$ that annihilates $H^\infty(\bD)+\rC(\bT)$. Since $\rC(\bT)$ is weak-$*$ dense in $L^\infty(\bT,\lambda)$, we see that $\phi$ is not weak-$*$ continuous, so that $H^\infty(\bD)$ does not have the strong F.\& M. Riesz property in $L^\infty(\bT,\lambda)$ with respect to $\Delta$. On the other hand,  it is known
 that $H^\infty(\bD)$ has  the F.\& M. Riesz property in $L^\infty(\bT,\lambda)$ with respect to $\Delta$ (see \cite{ando1978}). We may thus use Corollary \ref{C:FMs} to see that the  set $E(\Delta)$  is not a left
band. \qed
\end{example}

\subsection{Henkin functionals and the Cole-Range phenomenon}
We now wish to connect Corollary \ref{C:FMs} with the main objects studied in \cite{CT2023Henkin}. This will be accomplished via an appropriate choice of $\Delta$.

Let $M$ be a von Neumann algebra. Let $B\subset M$ be a unital $\rC^*$-subalgebra that is weak-$*$ dense in $M$. Let $A\subset B$ be a unital subalgebra.  A bounded linear functional $\phi:A\to\bC$ is said to be \emph{Henkin} relative to $M$ if it extends weak-$*$ continuously to the weak-$*$ closure of $A$ in $M$. This notion is in agreement with our terminology used in Section \ref{S:A(H)}. 
Equivalently, $\phi$ is Henkin  if it extends weak-$*$ continuously to $M$ \cite[Theorem 3.6]{rudin1991FA}, so the space of Henkin functionals is simply $(M_*)|_A$.

Let $\Delta$ denote the set of restrictions to $B$ of normal states on $M$.  Let $X\subset B^*$ denote the subspace consisting of functionals $\phi\in B^*$ such that $\phi|_A$ is Henkin -- this is precisely the set considered in Theorem \ref{C:Lebdecomp}. This subspace is closed whenever $(M_*)|_A$ is closed, which can be ascertained with the help of Lemma \ref{L:Henkinclosed}. Further, by virtue of Lemma \ref{L:restnormal}, we see that $X=E(\Delta)$. In particular, Corollary \ref{C:FMs} can serve as a tool for checking when $X$ forms a left band,
which was a central question in \cite{CT2023Henkin}, and shown to be useful in Corollary \ref{C:Lebdecomp}. 

Furthermore, Theorem \ref{T:colerange}  can  be interpreted as a unifying principle underlying previous known results, as we illustrate next.

\begin{example}\label{E:Henkin}
Assume that $X$, as defined in the last paragraph, is norm-closed (see Lemma  \ref{L:Henkinclosed}).
Clearly, $A^\perp\subset X$. Next, let $\Omega\subset B^*$ denote  the norm closure of the convex hull of $\{|\phi|:\phi\in A^\perp,\|\phi\|=1\}$. In this case, $A^\perp\subset \AC(\Omega)$.
If $\psi \in E(\Omega)$ then as before using the definition of $E(\Omega)$ we may  find 
 $\nu\in A^\perp$ and $\alpha\in \AC(\Omega)$ such that $\psi =\alpha+\nu \in \AC(\Omega)$.   Thus 
 $E(\Omega)=\AC(\Omega)$.
 Clearly $A$ has the strong F.\& M. Riesz property in $B$ with respect to $\Omega$.

Let us now examine the assumption $\Omega\subset X\subset E(\Omega)$ from Theorem \ref{T:colerange}.
The inclusion $X\subset E(\Omega)$ is equivalent to $X|_A\subset \AC(\Omega)|_A$, which can then be phrased in terms of weak topologies on $A$, 
matching  the condition of \emph{analyticity} in \cite[Theorem 4.4]{CT2023Henkin}, and
recovering that theorem. 
Next, since  $A^\perp \subset X$, if $X$ is a left band  then $\Omega\subset X$.
This follows because  if $\varphi \in A^\perp \subset X$ then $| \phi | \in X$, and because $X$ is convex and assumed to be closed. 
Theorem \ref{T:colerange}  implies in turn that $X=\AC(\Omega)$, assuming analyticity. This gives a 
variant of \cite[Theorem 4.5]{CT2023Henkin}.
\qed
\end{example}

\begin{example}\label{E:AB}
We go back to the setting of Example \ref{E:ballalg}, so that $M=L^\infty(\bS_d,\sigma)$, $B=\rC(\bS_d)$ and $A$ is the ball algebra. Again, $X\subset B^*$ is the space of functionals whose restriction to $A$ is Henkin relative to $M$.  Next, let $\Omega\subset B^*$ denote the norm-closed convex set of states $\rho$ such that 
\[
\rho(a)=a(0),\quad a\in A.
\]
Clearly,  integration against $\sigma$ is in $\Omega$ by Cauchy's formula. Furthermore,  $\Omega \subset X$. 
Since Henkin functionals on $A$  admit a weak-$*$ continuous extension to $M$, it follows that $X\subset A^\perp+\AC(\sigma)\subset A^\perp+\AC(\Omega)=E(\Omega)$. 

It is known that $A$ has the F.\& M. Riesz property in $B$ with respect to $\Omega$ by \cite[Theorem 9.5.6]{rudin2008}. In addition, $X$ is a left band by Henkin's theorem \cite[Theorem 9.3.1]{rudin2008}. Hence, applying Theorem \ref{T:colerange}, we find $X=\AC(\Omega)$, and we recover the classical characterization of Henkin measures for the ball algebra due to Cole and Range \cite[Theorem 9.6.1]{rudin2008}. \qed
\end{example}

\section{The Amar--Lederer property}\label{S:AL}

Section \ref{S:FM} studied the interplay between the F.\& M. Riesz properties as a mechanism for understanding when all the statements appearing in Theorem \ref{T:ForelliFM} are equivalent. In this section, we tackle this question in different direction, focusing instead on the implication (vi) $\Rightarrow$ (iv) therein. As we shall see, this is a rather delicate issue.

Let $B$ be a unital $\rC^*$-algebra and let $A\subset B$ be a unital norm-closed subalgebra. Let $\Delta$ be a norm-closed convex subset of the state space of $B$. We say that $A$ has the \emph{Amar--Lederer property in $B$ with respect to $\Delta$} if all closed projections in $\AC(\Delta)^\perp$ are dominated by an $A$-{\em peak} projection in $\AC(\Delta)^\perp$. The reason for this terminology is made clear in the next example, which incidentally also shows that the implication  (vi) implies (iv) generally fails in Theorem \ref{T:ForelliFM}, even for commutative $B$ and $\Delta$ a singleton, and even if we then strengthen (vi) to insist that
all projections there are closed.

\begin{example}\label{E:AL}
Consider $H^\infty(\bD)\subset L^\infty(\bT,\lambda)$ as in Example \ref{E:Hinfband}.  Let $X$ denote the character space of $L^\infty(\bT,\lambda)$. Then, via the Gelfand transform, we know that $L^\infty(\bT,\lambda)\cong \rC(X)$. Let $A\subset \rC(X)$ denote the unital norm-closed subalgebra corresponding to $H^\infty(\bD)$ under this identification. Let $\tau$ denote the state on $\rC(X)$ corresponding to integration against $\lambda$ on $L^\infty(\bT,\lambda)$. Let $\Delta=\{\tau\}$. In this case,  $A$ has the Amar--Lederer property in $B$ with respect to $\Delta$ by a classical result of Amar and Lederer \cite{AL1971}. 
Of course understanding this requires also the dictionary between closed sets and closed projections, and 
between  peak sets and closed projections in $A^{\perp \perp}$,
  that we have discussed earlier.  On the other hand, properties (iv) and (v) in Theorem \ref{T:ForelliFM} fail.
To see this note that the latter properties, in the specific case of this example,  imply that closed `Lebesgue-null' sets are generalized peak sets.
Yet  \cite[Proposition 6.3]{pelczynski1977} gives a   closed `Lebesgue-null' $G_\delta$ set in this context which is not a peak set. 
Hence, it is not a generalized peak set, since $G_\delta$ generalized peak sets are peak sets \cite[Lemma II.12.1]{gamelin1969}.  
\end{example}

For $H^\infty(\bD)\subset L^\infty(\bT,\lambda)$ as above, the Amar--Lederer property was established in \cite{AL1971} as a tool in characterizing the exposed points of the unit ball of $H^\infty(\bD)$. Subsequently, it also played a key role in Ando's proof of uniqueness for the predual of $H^\infty(\bD)$ \cite{ando1978}, and in establishing a Lebesgue decomposition for $H^\infty(\bD)$. In light of these developments, it stands to reason that the Amar--Lederer property should also be meaningful in the context of certain noncommutative analogues of the inclusion $H^\infty(\bD)\subset L^\infty(\bT,\lambda)$.  Indeed the Amar--Lederer property has inspired several deep papers in noncommutative function theory, some of them concerned with the topic of uniqueness of predual.  
  Ueda and the first author with Labuschagne have proved in \cite{ueda2009, BL2018} a partial Amar--Lederer theorem for maximal subdiagonal algebras in $\sigma$-finite von Neumann algebras, and used 
this to give several powerful applications such as a Lebesgue decomposition for such algebras.
We say `partial' because of a severe restriction imposed  in those papers on the projections in the theorem. 
Accordingly, we set out to remove this restriction  and establish the full noncommutative Amar--Lederer theorem for the class of  maximal subdiagonal algebras. 
The first part of argument has little to do with those algebras, but instead is  purely von Neumann algebraic in nature.

\subsection{Von Neumann algebraic preliminaries}

Let $B$ be a
 $\rC^*$-algebra and let $q\in B^{**}$ be a projection. Let $\phi$ be a state on $B$.
We say that $\phi$ is \emph{regular at $q$} if  for all projections $p\in B$ dominating $q$ and for all $\eps > 0$, there exists an open projection $u\in B^{**}$ such that $p \geq u\geq q$ and $\phi(u) < \phi(q) + \eps$.
(Note that we are writing here and below, as is common, $\hat{\phi}(u)$ as $\phi(u)$, where 
$\hat{\phi}$ is the extension of $\phi$ to the bidual.)

For example, let $X$ be a compact Hausdorff space and let $B=\rC(X)$. Let $\mu$ be regular Borel measure on $X$, and let $\phi$ be the state on $B$ of integration against $\mu$. Let $q=\chi_E\in B^{**}$ for some Borel subset $E\subset X$. Then, $\phi$ is regular at $q$. Indeed, given $\eps>0$, by the regularity property of the measure $\mu$, we may choose an open projection $u\in B^{**}$ such that $u_0\geq q$ and $\phi(u_0)<\phi(q)+\eps$. If we are also given a projection $p\in B$ such that $p\geq q$, then we may simply take $u=u_0p$, which is still open and has the required property.

\begin{proposition} \label{P:stateapproxopen} 
Let $B$ be a unital $\rC^*$-algebra and let $q\in B^{**}$ be a projection. Let $\phi$ be a state on $B$.
Two sufficient  conditions for  regularity of $\phi$  at $q$ are the following: 
\begin{enumerate}[{\rm (i)}]
\item $q$ is open or closed.
\item $q$ is of type $F_\sigma$ and $\phi$ is tracial with $\phi(q) = 0$.
\end{enumerate}
\end{proposition}
\begin{proof}
(i)  The `open' case is trivial, so assume $q$ closed. We adapt
 \cite[Proposition II.3]{akemann1969}.  Write $q^\perp=1-q$. Let $\eps>0$ and let $p$ be a projection in $B$ with $p \leq q^\perp$. The desired conclusion is equivalent to showing that there is a closed projection $r\in B^{**}$ such that $p\leq r\leq q^\perp$ and $\phi(r)>\phi(q^\perp)-\eps$.  Because $q$ is closed and is orthogonal to $p$, $p+q$ is closed \cite[Theorem II.7]{akemann1969}, and hence $q^\perp - p=1-(q+p)$ is open. Thus, there is an increasing net of contractions $(b_i)$ in $B$ converging to $q^\perp - p$ in the weak-$*$ topology of $B^{**}$. Consider the positive element $a_i = b_i + p\in B$. The increasing net $(a_i)$ converges to $q^\perp$ in the weak-$*$ topology of $B^{**}$. Choose an index $i$ large enough so that $\phi(q^\perp-a_i)<\eps/2$. Note that $a_i p = p$ since $p b_i p \leq p (q^\perp - p) p = 0$.  Thus for any polynomial $f$ we have $f(a_i) p = f(1) p$, and by uniform approximation the same is true for any continuous function on $[0,1]$. The proof of  \cite[Proposition II.3]{akemann1969} constructs a sequence $(f_n)$ of continuous functions on $[0,1]$ with $f_n(1)=1$ such that $(f_n(a_i))$ increases in the weak-$*$ topology to the closed projection $r=\chi_{[\eps/2,1]}(a_i)\in B^{**}$. Because $f_n(a_i)p=p$ for each $n$, we find $rp=p$ or $r\geq p$. Finally, arguing as in  \cite[Proposition II.3]{akemann1969}, we see that $\phi(q^\perp-r)<\eps$ as desired.

(ii) Since $q$ is of type $F_\sigma$, there is a countable collection $\{q_k:k\geq 1\}$ of closed projections in $B^{**}$ such that $q=\bigvee_{k} q_k$. Let $\eps>0$, and 
 let $p$ be a projection in $B$ dominating $q$. For each $k\geq 1$, we apply (i) to find an open projection $u_k\in B^{**}$ 
 with $\phi(u_k) <\eps/2^k$ and $q_k \leq u_k \leq p$. The projection $u=\bigvee_k u_k$ is open and clearly dominates $q$
 and is dominated by $p$. Furthermore, since $\phi$ is tracial on $B$, it is also tracial on $B^{**}$ so we see that $ \phi(u)\leq\sum_{k} \phi(u_k)$ by \cite[Proposition V.1.6]{takesaki2002}. Therefore $\phi(u)<\eps$.
\end{proof}

{\em Remark.} Proposition \ref{P:stateapproxopen} is true for nonunital $\rC^*$-algebras with `closed' in (i) replaced by `compact'.  This follows almost immediately by applying the unital case in the unitization $B^1$,
since  $q$ is compact if it is closed in $B^1$.

\medskip

We also record a fact found in the proof of \cite[Proposition II.14]{akemann1969}.

\begin{lemma}\label{L:openvN}
Let $M$ be a von Neumann algebra. Let $u\in M^{**}$ be an open projection and $\phi$ be a normal state on $M$. Then, there is a projection $p\in M$ dominating $u$ such that $\phi(u)=\phi(p)$.
\end{lemma}

We now establish the pivotal technical tool that we require.

\begin{proposition}\label{P:sequenceproj}
 Let $M$ be a von Neumann algebra, let $\phi$ be a normal state on $M$ and let $q\in M^{**}$ be a projection.
 Then $\phi$ is regular at $q$ if and only if there is a decreasing sequence of projections $(p_n)$ in $M$ dominating $q$ such that $\lim_{n\to\infty}\phi(p_n)=\phi(q)$.
\end{proposition}
\begin{proof}  Suppose that  there is a decreasing sequence of projections $(p_n)$ in $M$ dominating $q$ such that $\lim_{n\to\infty}\phi(p_n)= \phi(q)$, and that 
$p$  is a projection in $M$ dominating $q$. For each $n$, we let $u_n=p\wedge p_n\in M$, which is clopen.  Since $q \leq u_n  \leq p_n$ we have that $\lim_{n\to\infty} \phi(u_n) =\phi(q)$.

For the other direction, put $p_0=1$.  Assume that we have constructed  projections $p_0,\ldots,p_n\in M$ dominating $q$ such that $p_j\leq p_k$ and $\phi(p_j)\leq \phi(q) + 1/(j+1)$ for every $0\leq k\leq j\leq n$. Consider the von Neumann algebra $N=p_n M p_n$. By regularity, there is an open projection $u = p_n u p_n \in M^{**}$ such that $u\geq q$ and $\phi(u)< \phi(q) + 1/(n+2)$.   Also $u$ is open in $N^{**}$.
Apply  Lemma \ref{L:openvN} to find a projection $p_{n+1}\in N$ such that $p_{n+1}\geq 
u \geq q$ and $\phi(p_{n+1})=\phi(u) < \phi(q) + 1/(n+2)$. Since $p_{n+1}\in N$, we must have that $p_{n+1}\leq p_n$. The existence of the desired sequence then follows by induction.
\end{proof}

\subsection{Subdiagonal algebras}

Let us recall the setting of Arveson's subdiagonal algebras,  referring the reader to \cite{BL2007, ueda2009, BL2018} for additional detail and references. 

Let $M$ be a $\sigma$-finite von Neumann algebra with a faithful normal state $\phi$.   Let $A\subset M$ be a unital weak-$*$ closed subalgebra such that $A+A^*$ is weak-$*$ dense in $M$. Put $D=A\cap A^*$ and let $E:M\to D$ be a (normal faithful) conditional expectation. We say that $A$ is \emph{subdiagonal} in $M$ with respect to $E$ if $E$ is multiplicative on $A$.  Such an algebra
$A$ is   said to be a {\em maximal subdiagonal algebra} if $A$ is properly contained in no larger subdiagonal algebra in $M$ with respect to $E$. 

We now motivate the main result of this section.  Let $\Delta\subset M^*$ denote the set of normal states on $M$. Recall that there is a central projection $z\in M^{**}$ such that $M_*=M^*z$ \cite[Theorem III.2.14]{takesaki2002}.  Let $\psi$ be a state on $M$ and consider its support projection $s_\psi\in M^{**}$, which is defined by
\[
\{x\in M^{**}:\psi(x^*x)=0\}=M^{**}(1-s_\psi).
\]
If we assume that $\psi\in \SG(\Delta)$, then by definition $s_\psi r_\Delta=0$ or $s_\psi\in \AC(\Delta)^\perp$, which by Lemma \ref{L:restnormal} is equivalent to $s_\psi\leq 1-z$. Under the additional requirement that the given faithful normal state $\phi$ be tracial, Ueda showed in \cite{ueda2009} that there always exists a closed projection $p\in \AC(\Delta)^\perp\cap A^{\perp\perp}$ such that $s_\psi\leq p$. 
This is a `partial Amar--Lederer theorem'.  At the cost of great technical difficulty, the tracial condition on $\phi$ was removed in \cite{BL2018}.

We now improve this result, where we replace the support projection of a singular state by any $\phi$-null regular projection. Fortunately, the very lengthy and 
delicate argument found in \cite{BL2018} can be adapted verbatim to our situation, so we simply need to verify that the required conditions for its application are met.

\begin{theorem}  \label{T:ALstate}   
Let $M$ be a $\sigma$-finite von Neumann algebra with a faithful normal state $\phi$. Let $\Delta$ denote the set of normal states on $M$. Let $A\subset M$ be a maximal subdiagonal algebra. Then, the following statements hold.
\begin{enumerate}[{\rm (i)}]
\item $A$ has the  Amar--Lederer property in $M$ with respect to $\Delta$.  
Indeed every regular projection that is $\phi$-null (or in $\AC(\Delta)^\perp$) is dominated by a closed peak 
projection in $\AC(\Delta)^\perp\cap A^{\perp\perp}$. 
\item $A$ has the  Forelli property in $M$ with respect to $\Delta$.
\end{enumerate}
\end{theorem}
 \begin{proof}  
 (i) 
Note that $\AC(\Delta)=\AC(\phi)$ by virtue of Lemmas \ref{L:restnormal} and \ref{L:predualphi}. In particular, it follows from \cite[Lemma 2.4]{CT2023Henkin} that a projection $p\in M^{**}$  lies in $\AC(\Delta)^{\perp}$ if and only if $\phi(p)=0$. Invoking Proposition \ref{P:stateapproxopen}, it thus suffices to fix a given projection $q\in M^{**}$ such that $\phi(q)=0$ and with the property that $\phi$ is regular at $q$, and to show that there then exists a closed projection $p\in M^{**}$ lying in $A^{\perp\perp}$ such that $\phi(p)=0$ and $p\geq q$.   By Proposition \ref{P:sequenceproj}
 there is a decreasing sequence of projections $(p_n)$ in $M$ dominating $q$ such that $\lim_{n\to\infty}\phi(p_n)=0$.  It follows that $p_n \searrow 0$ weak* in $M$ (since 
 if $p_n \searrow r$ then $\phi(r) = 0$ and so $r = 0$).  Let $p_0$ be the 
 infimum (and weak* limit) of the $p_n$  in $M^{**}$, a closed projection with $\phi(p_0) = 0$.   
 
 We now have completed our generalization of the first (von Neumann algebraic) part of the main theorem in \cite{ueda2009} or \cite{BL2018}.   Fortunately 
the above is all that is needed to be changed in the rest of the proofs of the latter.  
By the very technical proof comprising pages 8223--8228 in \cite{BL2018} we obtain 
 a contraction  $b \in A$ with $b p_0 = p_0 = p_0 b$.   By the basic theory of peak projections \cite{hay2007,blecher2013} the real positive element 
  $a = (1+b)/2$ peaks at a projection $p \geq p_0$, with $a^n \to p$ weak* in $M^{**}$.
  As in \cite{ueda2009} we have $a^n \to 0$ weak* in $M$, so that $p \in \AC(\Delta)^\perp \cap A^{\perp\perp}$.  By construction $p$ is a peak projection. 

For the last assertion, simply note that (vi) in Theorem \ref{T:ForelliFM}  implies the  Forelli property.
\end{proof}

Thus for example if we assume in addition that $\phi$ is tracial then every type $F_\sigma$ projection in $\AC(\Delta)^\perp$ is dominated by a closed 
projection in $\AC(\Delta)^\perp\cap A^{\perp\perp}$.

As noted in the proof we may replace  $\Delta$ by $\{ \phi \}$ with no change in the conclusions. 

Finally, we remark that statement (i) above 
 is almost certainly not provable for general (even commutative) von Neumann algebras.  Indeed there is a set theoretic
obstruction as is noted in the last section of \cite{BL2018}.


\medskip

We thank Ken Davidson and Michael Hartz for several   helpful discussions on the topic of peak interpolation and their work on Hilbert function algebras \cite{DH2023}.

\bibliography{smallproj} 

\def\cprime{$'$}
\begin{bibdiv}
\begin{biblist}

\bib{akemann1969}{article}{
      author={Akemann, Charles~A.},
       title={The general {S}tone-{W}eierstrass problem},
        date={1969},
     journal={J. Functional Analysis},
      volume={4},
       pages={277\ndash 294},
         url={https://doi-org.uml.idm.oclc.org/10.1016/0022-1236(69)90015-9},
      review={\MR{0251545}},
}

\bib{akemann1970left}{article}{
      author={Akemann, Charles~A.},
       title={Left ideal structure of {$C\sp*$}-algebras},
        date={1970},
     journal={J. Functional Analysis},
      volume={6},
       pages={305\ndash 317},
         url={https://doi-org.uml.idm.oclc.org/10.1016/0022-1236(70)90063-7},
      review={\MR{0275177}},
}

\bib{AAP1989}{article}{
      author={Akemann, Charles~A.},
      author={Anderson, Joel},
      author={Pedersen, Gert~K.},
       title={Approaching infinity in {$C^*$}-algebras},
        date={1989},
        ISSN={0379-4024},
     journal={J. Operator Theory},
      volume={21},
      number={2},
       pages={255\ndash 271},
      review={\MR{1023315}},
}

\bib{AHMR2019interp}{article}{
      author={Aleman, Alexandru},
      author={Hartz, Michael},
      author={McCarthy, John~E.},
      author={Richter, Stefan},
       title={Interpolating sequences in spaces with the complete pick
  property},
        date={2019},
        ISSN={1073-7928},
     journal={Int. Math. Res. Not. IMRN},
      number={12},
       pages={3832\ndash 3854},
         url={https://doi-org.uml.idm.oclc.org/10.1093/imrn/rnx237},
      review={\MR{3973111}},
}

\bib{AL1971}{article}{
      author={Amar, \'{E}ric},
      author={Lederer, Aline},
       title={Points expos\'{e}s de la boule unit\'{e} de {$H\sp{\infty
  }(D)$}},
        date={1971},
        ISSN={0151-0509},
     journal={C. R. Acad. Sci. Paris S\'{e}r. A-B},
      volume={272},
       pages={A1449\ndash A1452},
      review={\MR{283557}},
}

\bib{ando1978}{article}{
      author={Ando, T.},
       title={On the predual of {$H^{\infty }$}},
        date={1978},
     journal={Comment. Math. Special Issue},
      volume={1},
       pages={33\ndash 40},
        note={Special issue dedicated to W{\l}adys{\l}aw Orlicz on the occasion
  of his seventy-fifth birthday},
      review={\MR{504151}},
}

\bib{AB1980}{article}{
      author={Archbold, R.~J.},
      author={Batty, C. J.~K.},
       title={{$C\sp{\ast} $}-tensor norms and slice maps},
        date={1980},
        ISSN={0024-6107,1469-7750},
     journal={J. London Math. Soc. (2)},
      volume={22},
      number={1},
       pages={127\ndash 138},
         url={https://doi.org/10.1112/jlms/s2-22.1.127},
      review={\MR{579816}},
}

\bib{BC1966}{article}{
      author={Bade, William~G.},
      author={Curtis, Philip~C., Jr.},
       title={Embedding theorems for commutative {B}anach algebras},
        date={1966},
        ISSN={0030-8730,1945-5844},
     journal={Pacific J. Math.},
      volume={18},
       pages={391\ndash 409},
         url={http://projecteuclid.org/euclid.pjm/1102994122},
      review={\MR{202001}},
}

\bib{bishop1959bdry}{article}{
      author={Bishop, Errett},
       title={A minimal boundary for function algebras},
        date={1959},
        ISSN={0030-8730},
     journal={Pacific J. Math.},
      volume={9},
       pages={629\ndash 642},
         url={http://projecteuclid.org.uml.idm.oclc.org/euclid.pjm/1103039103},
      review={\MR{109305}},
}

\bib{bishop1962}{article}{
      author={Bishop, Errett},
       title={A general {R}udin-{C}arleson theorem},
        date={1962},
        ISSN={0002-9939,1088-6826},
     journal={Proc. Amer. Math. Soc.},
      volume={13},
       pages={140\ndash 143},
         url={https://doi.org/10.2307/2033792},
      review={\MR{133462}},
}

\bib{bishop1959}{article}{
      author={Bishop, Errett},
      author={de~Leeuw, Karel},
       title={The representations of linear functionals by measures on sets of
  extreme points},
        date={1959},
        ISSN={0373-0956},
     journal={Ann. Inst. Fourier. Grenoble},
      volume={9},
       pages={305\ndash 331},
      review={\MR{0114118 (22 \#4945)}},
}

\bib{blecher2013}{article}{
      author={Blecher, David~P.},
       title={Noncommutative peak interpolation revisited},
        date={2013},
        ISSN={0024-6093},
     journal={Bull. Lond. Math. Soc.},
      volume={45},
      number={5},
       pages={1100\ndash 1106},
         url={http://dx.doi.org.uml.idm.oclc.org/10.1112/blms/bdt040},
      review={\MR{3105002}},
}

\bib{blecher2022}{article}{
      author={Blecher, David~P.},
       title={A noncommutative {B}ishop peak interpolation-set theorem},
        date={2023},
     journal={Oper. Theory Adv. Appl.},
      volume={292},
       pages={61\ndash 69},
         url={https://doi.org/10.1007/978-3-031-38020-4_3},
      review={\MR{4696663}},
}

\bib{BHN2008}{article}{
      author={Blecher, David~P.},
      author={Hay, Damon~M.},
      author={Neal, Matthew},
       title={Hereditary subalgebras of operator algebras},
        date={2008},
        ISSN={0379-4024},
     journal={J. Operator Theory},
      volume={59},
      number={2},
       pages={333\ndash 357},
      review={\MR{2411049}},
}

\bib{BL2018}{article}{
      author={Blecher, David~P.},
      author={Labuschagne, Louis},
       title={Ueda's peak set theorem for general von {N}eumann algebras},
        date={2018},
        ISSN={0002-9947,1088-6850},
     journal={Trans. Amer. Math. Soc.},
      volume={370},
      number={11},
       pages={8215\ndash 8236},
         url={https://doi.org/10.1090/tran/7275},
      review={\MR{3852463}},
}

\bib{BL2007}{article}{
      author={Blecher, David~P.},
      author={Labuschagne, Louis~E.},
       title={Noncommutative function theory and unique extensions},
        date={2007},
        ISSN={0039-3223},
     journal={Studia Math.},
      volume={178},
      number={2},
       pages={177\ndash 195},
         url={https://doi-org.uml.idm.oclc.org/10.4064/sm178-2-4},
      review={\MR{2285438}},
}

\bib{BLM2004}{book}{
      author={Blecher, David~P.},
      author={Le~Merdy, Christian},
       title={Operator algebras and their modules---an operator space
  approach},
      series={London Mathematical Society Monographs. New Series},
   publisher={The Clarendon Press, Oxford University Press, Oxford},
        date={2004},
      volume={30},
        ISBN={0-19-852659-8},
  url={http://dx.doi.org.uml.idm.oclc.org/10.1093/acprof:oso/9780198526599.001.0001},
        note={Oxford Science Publications},
      review={\MR{2111973}},
}

\bib{BN2012}{article}{
      author={Blecher, David~P.},
      author={Neal, Matthew},
       title={Open projections in operator algebras {II}: {C}ompact
  projections},
        date={2012},
        ISSN={0039-3223,1730-6337},
     journal={Studia Math.},
      volume={209},
      number={3},
       pages={203\ndash 224},
         url={https://doi.org/10.4064/sm209-3-1},
      review={\MR{2944468}},
}

\bib{BR2011}{article}{
      author={Blecher, David~P.},
      author={Read, Charles~John},
       title={Operator algebras with contractive approximate identities},
        date={2011},
        ISSN={0022-1236},
     journal={J. Funct. Anal.},
      volume={261},
      number={1},
       pages={188\ndash 217},
         url={http://dx.doi.org.uml.idm.oclc.org/10.1016/j.jfa.2011.02.019},
      review={\MR{2785898}},
}

\bib{BR2013}{article}{
      author={Blecher, David~P.},
      author={Read, Charles~John},
       title={Operator algebras with contractive approximate identities, {II}},
        date={2013},
        ISSN={0022-1236},
     journal={J. Funct. Anal.},
      volume={264},
      number={4},
       pages={1049\ndash 1067},
         url={http://dx.doi.org.uml.idm.oclc.org/10.1016/j.jfa.2012.11.013},
      review={\MR{3004957}},
}

\bib{brown2014}{article}{
      author={Brown, Lawrence~G.},
       title={Large {$C^*$}-algebras of universally measurable operators},
        date={2014},
        ISSN={0033-5606,1464-3847},
     journal={Q. J. Math.},
      volume={65},
      number={3},
       pages={851\ndash 855},
         url={https://doi.org/10.1093/qmath/hat058},
      review={\MR{3261970}},
}

\bib{BO2008}{book}{
      author={Brown, N.P},
      author={Ozawa, N.},
       title={${C}^*$-algebras and finite-dimensional approximations},
      series={Graduate Studies in Mathematics},
   publisher={American Mathematical Society},
     address={Providence, RI},
        date={2008},
      volume={88},
}

\bib{christensen1981}{article}{
      author={Christensen, Erik},
       title={On nonselfadjoint representations of {$C^{\ast} $}-algebras},
        date={1981},
        ISSN={0002-9327},
     journal={Amer. J. Math.},
      volume={103},
      number={5},
       pages={817\ndash 833},
         url={http://dx.doi.org/10.2307/2374248},
      review={\MR{630768 (82k:46085)}},
}

\bib{clouatre2023pure}{article}{
      author={Clou\^atre, Rapha\"el},
       title={Restrictions of pure states to subspaces of {C}*-algebras},
        date={2025},
     journal={preprint arXiv:2306.02365, to appear in Journal of Functional
  Analysis},
}

\bib{CD2016duality}{article}{
      author={Clou\^atre, Rapha\"el},
      author={Davidson, Kenneth~R.},
       title={Duality, convexity and peak interpolation in the
  {D}rury-{A}rveson space},
        date={2016},
        ISSN={0001-8708},
     journal={Adv. Math.},
      volume={295},
       pages={90\ndash 149},
         url={http://dx.doi.org.uml.idm.oclc.org/10.1016/j.aim.2016.02.035},
      review={\MR{3488033}},
}

\bib{CH2025}{article}{
      author={Clou\^atre, Rapha\"el},
      author={Hartz, Michael},
       title={Lebesgue decompositions and the {G}leason-{W}hitney property for
  operator algebras},
        date={2025},
        ISSN={0022-2518,1943-5258},
     journal={Indiana Univ. Math. J.},
      volume={74},
      number={1},
       pages={157\ndash 193},
      review={\MR{4883588}},
}

\bib{CTh2022}{article}{
      author={Clou\^{a}tre, Rapha\"{e}l},
      author={Thompson, Ian},
       title={Finite dimensionality in the non-commutative {C}hoquet boundary:
  peaking phenomena and {$\rm C^*$}-liminality},
        date={2022},
        ISSN={1073-7928,1687-0247},
     journal={Int. Math. Res. Not. IMRN},
      number={20},
       pages={16046\ndash 16093},
         url={https://doi.org/10.1093/imrn/rnab184},
      review={\MR{4498170}},
}

\bib{CTh2023}{article}{
      author={Clou\^{a}tre, Rapha\"{e}l},
      author={Thompson, Ian},
       title={Minimal boundaries for operator algebras},
        date={2023},
        ISSN={2330-0000},
     journal={Trans. Amer. Math. Soc. Ser. B},
      volume={10},
       pages={807\ndash 832},
         url={https://doi.org/10.1090/btran/154},
      review={\MR{4611189}},
}

\bib{CT2023Henkin}{article}{
      author={Clou\^{a}tre, Rapha\"{e}l},
      author={Timko, Edward~J.},
       title={Non-commutative measure theory: {H}enkin and analytic functionals
  on {$\rm C^*$}-algebras},
        date={2023},
        ISSN={0025-5831,1432-1807},
     journal={Math. Ann.},
      volume={386},
      number={1-2},
       pages={415\ndash 453},
         url={https://doi.org/10.1007/s00208-022-02392-x},
      review={\MR{4585153}},
}

\bib{combes1970}{article}{
      author={Combes, Fran\c{c}ois},
       title={Quelques propri\'{e}t\'{e}s des {$C\sp{\ast} $}-alg\`ebres},
        date={1970},
        ISSN={0007-4497},
     journal={Bull. Sci. Math. (2)},
      volume={94},
       pages={165\ndash 192},
      review={\MR{306931}},
}

\bib{DH2023}{article}{
      author={Davidson, Kenneth~R.},
      author={Hartz, Michael},
       title={Interpolation and duality in algebras of multipliers on the
  ball},
        date={2023},
        ISSN={1435-9855,1435-9863},
     journal={J. Eur. Math. Soc. (JEMS)},
      volume={25},
      number={6},
       pages={2391\ndash 2434},
         url={https://doi.org/10.4171/jems/1245},
      review={\MR{4592872}},
}

\bib{douglas1998}{book}{
      author={Douglas, Ronald~G.},
       title={Banach algebra techniques in operator theory},
     edition={Second},
      series={Graduate Texts in Mathematics},
   publisher={Springer-Verlag, New York},
        date={1998},
      volume={179},
        ISBN={0-387-98377-5},
         url={https://doi-org.uml.idm.oclc.org/10.1007/978-1-4612-1656-8},
      review={\MR{1634900}},
}

\bib{gamelin1969}{book}{
      author={Gamelin, Theodore~W.},
       title={Uniform algebras},
   publisher={Prentice-Hall, Inc., Englewood Cliffs, N. J.},
        date={1969},
      review={\MR{0410387}},
}

\bib{GHX04}{article}{
      author={Guo, Kunyu},
      author={Hu, Junyun},
      author={Xu, Xianmin},
       title={{Toeplitz} algebras, subnormal tuples and rigidity on reproducing
  {$\bold C[z_1,\dots,z_d]$}-modules},
        date={2004},
        ISSN={0022-1236},
     journal={J. Funct. Anal.},
      volume={210},
      number={1},
       pages={214\ndash 247},
      review={\MR{2052120 (2005a:47007)}},
}

\bib{hay2007}{article}{
      author={Hay, Damon~M.},
       title={Closed projections and peak interpolation for operator algebras},
        date={2007},
        ISSN={0378-620X},
     journal={Integral Equations Operator Theory},
      volume={57},
      number={4},
       pages={491\ndash 512},
         url={http://dx.doi.org.uml.idm.oclc.org/10.1007/s00020-006-1471-z},
      review={\MR{2313282}},
}

\bib{LT2012}{article}{
      author={Levene, Rupert~H.},
      author={Timoney, Richard~M.},
       title={Completely bounded norms of right module maps},
        date={2012},
        ISSN={0024-3795,1873-1856},
     journal={Linear Algebra Appl.},
      volume={436},
      number={5},
       pages={1406\ndash 1424},
         url={https://doi.org/10.1016/j.laa.2011.08.036},
      review={\MR{2890927}},
}

\bib{pedersen1979book}{book}{
      author={Pedersen, Gert~K.},
       title={{$C^{\ast} $}-algebras and their automorphism groups},
      series={London Mathematical Society Monographs},
   publisher={Academic Press, Inc. [Harcourt Brace Jovanovich, Publishers],
  London-New York},
        date={1979},
      volume={14},
        ISBN={0-12-549450-5},
      review={\MR{548006}},
}

\bib{pelczynski1977}{article}{
      author={Pe{\l}czy\'{n}ski, Aleksander},
       title={Banach spaces of analytic functions and absolutely summing
  operators},
        date={1977},
      volume={No. 30},
       pages={i+91},
        note={Expository lectures from the CBMS Regional Conference held at
  Kent State University, Kent, Ohio, July 11--16, 1976},
      review={\MR{511811}},
}

\bib{rainwater1969}{article}{
      author={Rainwater, John},
       title={A note on the preceding paper},
        date={1969},
        ISSN={0012-7094},
     journal={Duke Math. J.},
      volume={36},
       pages={799\ndash 800},
         url={http://projecteuclid.org.uml.idm.oclc.org/euclid.dmj/1077378643},
      review={\MR{290114}},
}

\bib{rudin1991FA}{book}{
      author={Rudin, Walter},
       title={Functional analysis},
     edition={Second},
      series={International Series in Pure and Applied Mathematics},
   publisher={McGraw-Hill, Inc., New York},
        date={1991},
        ISBN={0-07-054236-8},
      review={\MR{1157815}},
}

\bib{rudin2008}{book}{
      author={Rudin, Walter},
       title={Function theory in the unit ball of {$\Bbb C^n$}},
      series={Classics in Mathematics},
   publisher={Springer-Verlag, Berlin},
        date={2008},
        ISBN={978-3-540-68272-1},
        note={Reprint of the 1980 edition},
      review={\MR{2446682 (2009g:32001)}},
}

\bib{shields1974}{incollection}{
      author={Shields, Allen~L.},
       title={Weighted shift operators and analytic function theory},
        date={1974},
   booktitle={Topics in operator theory},
      series={Math. Surveys},
      volume={No. 13},
   publisher={Amer. Math. Soc., Providence, RI},
       pages={49\ndash 128},
      review={\MR{361899}},
}

\bib{takesaki2002}{book}{
      author={Takesaki, M.},
       title={Theory of operator algebras. {I}},
      series={Encyclopaedia of Mathematical Sciences},
   publisher={Springer-Verlag, Berlin},
        date={2002},
      volume={124},
        ISBN={3-540-42248-X},
        note={Reprint of the first (1979) edition, Operator Algebras and
  Non-commutative Geometry, 5},
      review={\MR{1873025}},
}

\bib{ueda2009}{article}{
      author={Ueda, Yoshimichi},
       title={On peak phenomena for non-commutative {$H^\infty$}},
        date={2009},
        ISSN={0025-5831,1432-1807},
     journal={Math. Ann.},
      volume={343},
      number={2},
       pages={421\ndash 429},
         url={https://doi.org/10.1007/s00208-008-0277-5},
      review={\MR{2461260}},
}

\bib{varopoulos1971}{article}{
      author={Varopoulos, Nicholas~Th.},
       title={Sur la r\'{e}union de deux ensembles d'interpolation d'une
  alg\`ebre uniforme},
        date={1971},
        ISSN={0151-0509},
     journal={C. R. Acad. Sci. Paris S\'{e}r. A-B},
      volume={272},
       pages={A950\ndash A952},
      review={\MR{279591}},
}

\end{biblist}
\end{bibdiv}
\bibliographystyle{plain}

\end{document}